%% file: stokes1.tex
\pdfoutput=1
\documentclass[11pt]{article}
\include{stokes1-preamble}
\begin{document}

\title{{\huge\sffamily The Stokes groupoids}}
\author{Marco Gualtieri\footnote{\href{mailto:mgualt@math.toronto.edu}{\texttt{mgualt@math.toronto.edu}}} \and Songhao Li\footnote{\href{mailto:sli@math.toronto.edu}{\texttt{sli@math.toronto.edu}}} \and Brent Pym\footnote{\href{mailto:bpym@math.toronto.edu}{\texttt{bpym@math.toronto.edu}}}}
\date{}
\maketitle

\abstract{We construct and describe a family of groupoids over complex curves which serve as the universal domains of definition for solutions to linear ordinary differential equations with singularities.  As a consequence, we obtain a direct, functorial method for resumming formal solutions to such equations.}

\vspace{3em}
\makeatletter
\@starttoc{toc}
\makeatother
\vspace{3em}

\newcounter{airycount}
\setcounter{airycount}{1}
\pagebreak

\section{Introduction}
\input{stokes1-intro}

\section{Lie algebroids and meromorphic connections}
\label{sec:review}
\input{stokes1-review}

\section{Construction of Lie groupoids on curves}
\label{sec:constr}
\input{stokes1-constr}

\section{Source-simply connected groupoids}
\label{sec:unif}
\input{stokes1-unif}

\section{Extension of solutions over singularities}
\label{sec:solns}
\input{stokes1-solns}

\bibliographystyle{hyperamsplain}
\bibliography{stokes1}

\end{document}

%% file: stokes1-preamble.tex
\usepackage{amsthm,amsfonts,amssymb,amsmath,amscd} 
\usepackage{aliascnt} 
\usepackage{mathrsfs} 
\usepackage{xargs,ifthen} 
\usepackage{float}
\usepackage{comment}
\usepackage{color}
\usepackage{xypic} 
\usepackage{tikz}
\usetikzlibrary{matrix,arrows}
\usetikzlibrary{decorations.markings}
\usetikzlibrary{calc}
\usepackage[all,cmtip]{xy}
\usepackage{mathtools}
\usepackage{mathdots}
\usepackage{geometry}

\makeatletter
\renewenvironment{smallmatrix}[1][.2777em]{\null\,\vcenter\bgroup
  \Let@\restore@math@cr\default@tag
  \baselineskip6\ex@ \lineskip1.5\ex@ \lineskiplimit\lineskip
  \ialign\bgroup\hfil$\m@th\scriptstyle##$\hfil&&\kern#1\hfil
  $\m@th\scriptstyle##$\hfil\crcr
}{%
  \crcr\egroup\egroup\,%
}
\makeatother

\usepackage{hyperref}
\definecolor{tocolor}{rgb}{.1,.1,.1}
\definecolor{urlcolor}{rgb}{.2,.2,.6}
\definecolor{linkcolor}{rgb}{.1,.1,.5}
\definecolor{citecolor}{rgb}{.4,.2,.1}
\hypersetup{backref=true, colorlinks=true, urlcolor=urlcolor, linkcolor=linkcolor, citecolor=citecolor}

\newcommandx{\thdef}[2]{
	\newaliascnt{#1}{theorem}  
	\newtheorem{#1}[#1]{#2}
	\aliascntresetthe{#1}  
	\newtheorem*{#1*}{#2}
	\expandafter\newcommand\expandafter{\csname #1autorefname\endcsname}{#2}
}

\makeatletter
\newtheorem*{rep@theorem}{\rep@title}
\newcommand{\newreptheorem}[2]{%
\newenvironment{rep#1}[1]{%
 \def\rep@title{#2 \ref{##1}}%
 \begin{rep@theorem}}%
 {\end{rep@theorem}}}
\makeatother

\newtheorem{theorem}{Theorem}[section]
\newreptheorem{theorem}{Theorem}

\thdef{lemma}{Lemma}
\thdef{corollary}{Corollary}
\thdef{conjecture}{Conjecture}
\newreptheorem{conjecture}{Conjecture}
\thdef{proposition}{Proposition}

\theoremstyle{definition}
\thdef{definition}{Definition}
\thdef{notation}{Notation}
\thdef{example}{Example}
\theoremstyle{remark}
\thdef{remark}{Remark}

\theoremstyle{remark}

\newcommand{\defn}[1]{\textbf{\textit{#1}}} 

\newcommand{\spc}[1]{\mathsf{#1}} 
\newcommand{\shf}[1]{\mathcal{#1}} 

\newcommand{\RR}{\mathbb{R}}
\newcommand{\CC}{\mathbb{C}}
\newcommand{\DD}{\mathbb{D}}
\newcommand{\ZZ}{\mathbb{Z}}

\newcommand{\PP}[1][]{\mathbb{P}^{#1}}
\newcommand{\AF}[1][]{\mathbb{A}^{#1}}

\newcommand{\rbrac}[1]{\left(#1\right)} 

\newcommand{\mapdef}[5]{
	\begin{array}{ccccc}
	#1 &:& #2 &\to& #3 \\
		&&  #4 &\mapsto& #5
	\end{array}
}
\newcommandx{\fn}[2][2=]{#1\ifthenelse{\equal{#2}{}}{}{\!\rbrac{{#2}}}} 
\newcommandx{\id}[2][2=]{\fn{{\rm id}_{#1}}[#2]} 

\newcommand{\ext}[2][\bullet]{\spc{\Lambda}^{#1}{#2}} 
\newcommand{\sym}[2][\bullet]{\spc{S}^{#1}{#2}} 
\newcommandx{\End}[2][1=]{\fn{\spc{End}_{#1}}[#2]} 
\newcommandx{\Hom}[2][1=]{\fn{\spc{Hom}_{#1}}[#2]} 
\newcommandx{\Aut}[2][1=]{\fn{\spc{Aut}_{#1}}[#2]} 
\newcommandx{\image}[1]{\fn{\spc{img}}[#1]} 
\renewcommandx{\ker}[1]{\fn{\spc{ker}}[#1]} 
\newcommandx{\rank}[1]{\fn{\mathrm{rank}}[#1]} 

\newcommandx{\ann}[1]{\fn{\spc{ann}}[\spc{#1}]} 

\newcommandx{\hlgy}[3][1=\bullet,3=]{\spc{H}_{#1}^{#3}\!\rbrac{{#2}}} 
\newcommandx{\cohlgy}[3][1=\bullet,3=]{\spc{H}^{#1}_{#3}\!\rbrac{{#2}}} 
\newcommandx{\chow}[3][1=\bullet,3=]{\spc{A}^{#1}_{#3}\!\rbrac{{#2}}} 

\newcommandx{\Ext}[3][1=\bullet,3=]{\fn{\spc{Ext}^{#1}_{#3}}[{#2}]} 
\newcommandx{\Tor}[3][1=\bullet,3=]{\fn{\spc{Tor}^{#1}_{#3}}[{#2}]} 

\newcommandx{\Pic}[1]{\fn{\spc{Pic}}[{#1}]} 
\newcommandx{\chernalg}[2][1=\bullet]{\fn{\spc{Chern}^{#1}}[{#2}]} 
\newcommandx{\chern}[2][1=]{\fn{c_{#1}}[#2]} 
\newcommandx{\ch}[2][1=]{\fn{\mathrm{ch}_{#1}}[{#2}]} 

\newcommandx{\sKer}[2][1=]{ \fn{ \shf{K}er_{#1}}[{#2}] } 
\newcommandx{\sHom}[2][1=]{ \fn{ \shf{H}om_{#1}}[{#2}] } 
\newcommandx{\sEnd}[2][1=]{ \fn{ \shf{E}nd_{#1}}[{#2}] } 
\newcommandx{\sExt}[3][1=\bullet,3=]{\fn{\shf{E}xt^{#1}_{#3}}[{#2}]} 
\newcommandx{\sTor}[3][1=\bullet,3=]{\fn{\shf{T}or^{#1}_{#3}}[{#2}]} 
\newcommand{\sO}[1]{\shf{O}_{#1}}

\newcommandx{\forms}[2][1=\bullet]{\Omega^{#1}_{#2}} 
\newcommandx{\can}[1][1=]{\omega_{#1}} 
\newcommandx{\acan}[1][1=]{\omega_{#1}^{-1}} 
\newcommandx{\tshf}[1]{\shf{T}_{#1}} 
\newcommandx{\mvect}[2][1=\bullet]{ \ext[#1]{\tshf{#2}} }
\newcommandx{\der}[2][1=\bullet]{\mathscr{X}^{#1}_{#2}} 
\newcommandx{\sJet}[3][1=,2=]{\shf{J}^{#1}_{#2}#3} 

\newcommandx{\tb}[2][1=]{\spc{T}_{\!#1}{#2}} 
\newcommandx{\ctb}[2][1=]{\spc{T}_{\!#1}^*{#2}} 
\newcommandx{\lie}[2][2=]{\fn{\mathscr{L}_{#1}}[#2]} 
\newcommandx{\hook}[2][2=]{\fn{i_{#1}}[#2]} 
\newcommand{\cvf}[1]{\partial_{{#1}}} 
\newcommand{\deriv}[2]{\frac{d#1}{d#2}} 

\newcommand{\Airy}{\mathrm{Airy}}
\newcommand{\Ai}{\mathrm{Ai}}
\newcommand{\Bi}{\mathrm{Bi}}

\newcommand{\HH}{\mathbb{H}}
\renewcommand{\P}{\spc{P}} 

\newcommand{\X}{\spc{X}}

\newcommand{\Y}{\spc{Y}}
\newcommand{\Z}{\spc{Z}}
\newcommand{\U}{\spc{U}}

\newcommand{\V}{\spc{V}}
\newcommand{\C}{\spc{C}}
\newcommand{\D}{\spc{D}}
\newcommand{\R}{\spc{R}}

\renewcommand{\S}{\spc{S}}
\newcommand{\T}{\spc{T}}
\newcommand{\Sig}{\spc{\Sigma}}
\newcommand{\Tot}[1]{\fn{\spc{Tot}}[{#1}]}

\newcommand{\sD}{\shf{D}}
\newcommand{\sE}{\shf{E}}
\newcommand{\E}{\spc{E}}
\newcommand{\sL}{\shf{L}}
\newcommand{\sI}{\shf{I}} 
\newcommand{\m}{\mathfrak{m}} 

\newcommand{\sN}{\shf{N}}
\newcommand{\sV}[1]{\shf{V}#1} 
\newcommand{\sS}{\shf{S}}

\newcommand{\Par}[1]{\fn{\mathrm{Par}}[#1]}

\newcommand{\Bl}[2]{\fn{\spc{Bl}_{#2}}[#1]}

\newcommand{\EM}[2]{[#1\!:\!#2]} 

\newcommand{\sA}{\shf{A}}
\newcommand{\A}{\spc{A}}

\newcommand{\ati}[2][]{\fn{\shf{D}^{1}_{#1}}[#2]}

\newcommand{\FG}[1]{\fn{\spc{\Pi}_1}[{#1}]}
\newcommand{\GL}[1]{\fn{\spc{GL}}[#1]}

\newcommand{\PGL}[1]{\fn{\spc{PGL}}[#1]}
\newcommand{\iso}[2][1]{\fn{\mathfrak{iso}_{#1}}[#2]}

\newcommand{\Sto}[1]{\spc{Sto}_{#1}}
\newcommand{\Pair}[1]{\fn{\spc{Pair}}[{#1}]}
\newcommand{\F}{\spc{F}}
\newcommand{\G}{\spc{G}}
\renewcommand{\H}{\spc{H}}
\newcommand{\gid}{\id{}}

\newcommand{\LA}[1][]{\fn{\shf{L}ie}[#1]} 

\newcommand{\Mod}[1]{{#1}\mathrm{-mod}}

\newcommand{\Rep}[1]{\fn{\mathrm{Rep}}[#1]}

\newcommand{\airylabel}{The Airy equation, part \Roman{airycount}\addtocounter{airycount}{1}}

\newcommand{\mg}{\mathfrak{g}}
\newcommand{\red}{\mathrm{red}}
\newcommand{\sF}{\shf{F}}

\newcommand{\hX}{\widehat{\X}}
\newcommand{\hG}{\widehat{\G}}
\newcommand{\hsE}{\widehat{\sE}}
\newcommand{\hnabla}{\widehat{\nabla}}
\newcommand{\hphi}{\widehat{\phi}}
\newcommand{\hPsi}{\widehat{\Psi}}

\newcommand{\hf}{\!\hat{f}}
\newcommand{\mat}[1]{\begin{pmatrix}#1\end{pmatrix}}

%% file: stokes1-intro.tex

The Riemann--Hilbert correspondence, in its simplest form, is a natural equivalence between flat connections over a manifold and representations of its fundamental group.  

One approach to the correspondence, drawn schematically below, expresses it as a composition of two equivalences.  First, we view a flat connection on the vector bundle $\E$ over a manifold $\X$ as a representation of the Lie algebroid $\tshf{\X}$ of vector fields.  Solving the system of differential equations determined by the connection, we obtain parallel transport isomorphisms between the fibers of $\E$.  These parallel transports define a representation of the fundamental groupoid $\FG{\X}$, which is the set paths in $\X$ considered up to homotopy with fixed endpoints.  
Note that among the Lie groupoids with Lie algebroid $\tshf{\X}$, the fundamental groupoid is the ``largest'', being the unique source-simply connected one.
In this way, we obtain an equivalence between representations of $\tshf{\X}$ and representations of $\FG{\X}$, in direct analogy with the equivalence between the representation categories of a  finite-dimensional Lie algebra and its unique simply connected Lie group, a fact which follows from Lie's second theorem. 	     

The second equivalence, between the representation categories of the fundamental groupoid $\FG{\X}$ of a connected manifold $\X$ and the fundamental group $\pi_1(\X,p)$ based at a point $p\in \X$, is an instance of Morita equivalence of groupoids, and is simply given by restriction along the inclusion $\pi_1(\X,p)\hookrightarrow \FG{\X}$.  
\begin{equation*}
\begin{tikzpicture}[baseline=(current  bounding  box.center)]
\matrix (m) [matrix of math nodes, column sep=5em]
{ \Rep{\tshf{\X}} &\Rep{\FG{\X}}& \Rep{\pi_1(\X,p)}\\};
\path[latex'-latex']
(m-1-1) edge node [below] {$\cong$} (m-1-2);
\path[latex'-latex']
(m-1-2) edge node [below] {$\cong$} (m-1-3);
\end{tikzpicture}
\end{equation*}
From this point of view, the essential component of the Riemann--Hilbert map, which requires solving differential equations, is the passage from a representation of $\tshf{\X}$ to a representation of $\FG{\X}$.  In this paper, we investigate the generalization of this equivalence to the case where the connections in question are allowed to have singularities.  

Even when $\X$ is a smooth complex curve, the Riemann--Hilbert correspondence becomes much more interesting when the connections involved have poles along a divisor $\D$.  Because the existence and uniqueness theorem for solutions to ODEs does not apply at singular points, we no longer obtain a representation of $\FG{\X}$, and the connection is no longer characterized by a representation of $\pi_1(\X,p)$.  Instead, we obtain a representation of the fundamental group of the punctured curve $\X\setminus\D$, as well as local data related to the asymptotic behaviour of solutions at the poles.  This data includes Levelt filtrations in the case of regular singularities, as well as the Stokes data associated to irregular singularities.  

The main problem that we solve in this paper is to construct and describe the Lie groupoid $\FG{\X,\D}$ which replaces $\FG{\X}$ when poles bounded by the effective divisor $\D$ are allowed.  The importance of the Lie groupoid $\FG{\X,\D}$ is best explained by the fact that the fundamental solution of any meromorphic system with singularities bounded by $\D$, while singular along $\X$, is actually smooth when viewed as a function on the 2-dimensional complex manifold $\FG{\X,\D}$. In other words, the groupoid is the universal domain of definition for fundamental solutions to such systems.

One of the well-known features of meromorphic connections with irregular singularities is that the naive method of solving them in power series about the singular points usually leads to formal power series with zero radius of convergence.  Such a formal solution may be interpreted, in each of a collection of sectors surrounding the pole, as an asymptotic approximation to an actual solution to the system.  In sectorial overlaps, these solutions differ by Stokes factors.  Furthermore, there are methods, including those of Borel, \'Ecalle, and Ramis, by which the formal solution may be ``resummed'' to obtain actual solutions.  

As an application of the groupoid theory, we provide a new, more direct approach to the problem of resummation.  We show that the formal solution, when pulled back to the groupoid in the appropriate way, defines a power series in two variables which converges to a holomorphic representation.  This representation is a universal solution to the system, and may be used to obtain actual solutions in the usual sense.

\vspace{1em}
\noindent {\it Organization of the paper:}
\vspace{1ex}

In~\autoref{sec:review}, we describe the basic theory of Lie algebroids over complex curves and their representations, beginning with the simple observation that meromorphic connections with singularities bounded by an effective divisor $\D$ are nothing but holomorphic representations of the Lie algebroid $\tshf{\X}(-\D)$ that vanish on $\D$.

Sections~\ref{sec:constr} and~\ref{sec:unif} contain the main constructions: we introduce two ways to obtain groupoids integrating the algebroid $\tshf{\X}(-\D)$. The first is a blowup construction of the adjoint or ``smallest'' groupoid $\Pair{\X,\D}$ integrating $\tshf{\X}(-\D)$, based on the results of~\cite{Gualtieri2012}. The second method makes use of the uniformization of the punctured curve $\X\setminus\D$ to give an explicit description of $\FG{\X,\D}$ as an open set in a canonical projective line bundle over $\X$ equipped with a meromorphic Cartan projective connection.  Note that while $\Pair{\X,\D}$ is an algebraic Lie groupoid by construction, $\FG{\X,\D}$ is a holomorphic Lie groupoid which is algebraic only when $\X\setminus\D$ is simply connected.

Experts in meromorphic connections may wish to skip directly to the explicit examples of groupoids in~\autoref{sec:p1ex} and then to the application to the problem of resummation of formal power series in~\autoref{sec:solns}.

\vspace{1em}
\noindent {\it Acknowledgements:}
\vspace{1ex}

We thank Philip Boalch, Nigel Hitchin, Jacques Hurtubise, Alan Weinstein, Michael Wong and Peter Zograf for helpful discussions.  Marco Gualtieri was supported by an NSERC Discovery Grant and an Ontario ERA, Songhao Li was supported by an Ontario Graduate Scholarship, and Brent Pym was supported by an NSERC Canada Graduate Scholarship (Doctoral).

%% file: stokes1-review.tex

Throughout the paper, $\X$ is a smooth complex curve and $\D$ is an effective divisor on $\X$.  
A morphism of curves will be taken to mean a holomorphic map between curves that is nonconstant on each component of its domain, and the pullback of a divisor under such a morphism is defined by pulling back its local defining equations.


\subsection{Lie algebroids on curves}\label{sec:liealgcur}

Recall that a Lie algebroid on the manifold $\X$ is a locally free
sheaf $\sA$ equipped with a Lie bracket $[\cdot,\cdot]$ and a bracket-preserving anchor map $a:\sA\to\tshf{\X}$ such that
the Leibniz rule holds:
\begin{equation} [u,fv] = (\lie{a(u)}f)v + f[u,v],\qquad f \in
  \sO{\X}\text{ and }u,v \in \sA.
\end{equation}
If $\X$ is a curve and the anchor $a\in\cohlgy[0]{\X,\sA^\vee\otimes\tshf{\X}}$ is nonzero, then it vanishes along an effective
divisor $\D$, and hence lifts to a surjection $\sA\to \tshf{\X}(-\D)$.  If $\sA$ has rank $1$, we conclude that $\sA = \tshf{\X}(-\D)$, with anchor given by the inclusion
$\tshf{\X}(-\D)\hookrightarrow \tshf{\X}$, and bracket inherited from
the Lie bracket of vector fields.
\begin{proposition}
  Any Lie algebroid of rank $1$ over the curve $\X$ is either a bundle of abelian Lie algebras, or is canonically equivalent to the sheaf $\tshf{\X}(-\D)$ of vector fields vanishing on an effective divisor.
\end{proposition}

Over a fixed curve $\X$, the correspondence $\D\mapsto \tshf{\X}(-\D)$
is clearly functorial: a subdivisor $\D\leq \D'$ induces a Lie
algebroid morphism $\tshf{\X}(-\D')\to \tshf{\X}(-\D)$. The
correspondence is more interesting, however, when it involves different curves. If $f:\X\to\Y$ is a morphism of curves and
$\sA_\Y$ is a Lie algebroid over $\Y$, then the pullback algebroid
$f^!\sA_\Y$ is a Lie algebroid over $\X$ defined as the following
fiber product.
\begin{equation}\label{algpull}
  \begin{aligned}
    \xymatrix@R=1.5em@C=1.5em{f^!\sA_\Y\ar[d]\ar[r] &
      f^*\sA_\Y\ar[d]^-{a}\\ \tshf{\X}\ar[r]_-{Df} & f^*\tshf{\Y}}
  \end{aligned}
\end{equation}
The sheaves above may be described in terms of divisors as follows. First, if $\sA_\Y=\tshf{\Y}(-\D)$, then the pullback is $f^*\sA_\Y = (f^*\tshf{\Y})(-f^*\D)$.  If $\R$ is the ramification divisor of $f$, defined by the derivative $Df\in \cohlgy[0]{\X,\tshf{\X}^\vee\otimes f^*\tshf{\Y}}$, then $f^*\tshf{\Y} = \tshf{\X}(\R)$.  As a result, Diagram~\ref{algpull} implies the following.
\begin{proposition}\label{prop:pullie}
  Let $f:\X\to\Y$ be a morphism of curves and let $\sA_\Y
  = \tshf{\Y}(-\D)$ be the Lie algebroid associated to the effective divisor
   $(\Y,\D)$.  Then the Lie algebroid pullback $f^!\sA_\Y$ is given by
  \begin{equation}
    f^!\sA_\Y = f^*\sA_\Y(-\R\setminus f^*\D) = \tshf{\X}(-f^*\D\setminus\R),
  \end{equation}
  where $\R$ is the ramification divisor of $f$, and
  $\D_1\setminus\D_2$ is the divisor $\D_1 - \D_1\cap\D_2$.
\end{proposition}
If $\sA_\X$ is a Lie algebroid on $\X$, then a Lie algebroid
morphism $\sA_\X\to\sA_\Y$ over the map $f:\X\to\Y$ is a base-preserving Lie algebroid morphism $\sA_\X\to f^!\sA_\Y$.  As a result of the discussion above, we obtain the following criterion for existence of such a morphism.
\begin{corollary}\label{cor:mordiv}
Let $f:\X\to\Y$ be
a morphism of curves, and let $\C$ and $\D$ be effective divisors on $\X$ and $\Y$, respectively.  Then there exists a morphism of Lie algebroids
$\tshf{\X}(-\C)\to \tshf{\Y}(-\D)$ if and only if 
\begin{equation}
\C \geq f^*\D\setminus \R.
\end{equation}
Moreover, such a morphism, if it exists, is unique.
\end{corollary}

A special case occurs when the divisor $\D$ contains all critical 
values of $f$, or more precisely, when $\R\leq f^*\D$. Then $f^*\sA_\Y = f^!\sA_\Y$, and if $\C = f^*\D - \R$, we obtain an isomorphism $\sA_\X\to f^*\sA_\Y$. Such a situation is an algebroid analogue of a covering map, but where the underlying morphism of curves $f$ may be a branched covering.  This special case is interesting because it allows the pushforward of connections, as described in~\textsection\ref{funcmod}.
\begin{definition}\label{def:etale}
Let $f:\X\to\Y$ be a surjective morphism of curves equipped with effective divisors $\C,\D$. 
If the ramification divisor $\R$ satisfies $\R\leq f^*\D$ and $\C = f^*\D - \R$, we say that the morphism $\tshf{\X}(-\C)\to\tshf{\Y}(-\D)$ is \defn{\'{e}tale}.
\end{definition}
%
%
\begin{example}\label{ex:branchcover}
For any finite branched covering of curves $f:\X\to \Y$, let $\R$ be the ramification divisor, let $\Delta$ be the reduced branching divisor, i.e., the sum of the critical values of $f$, and let $f^{-1}(\Delta)$ be the (reduced) divisor defined by the sum of the preimages of the critical values. Then $f^*\Delta = \R + f^{-1}(\Delta)$, giving $\R\leq f^*\Delta$ in particular.  So, we obtain an \'{e}tale algebroid morphism 
\begin{equation}
\tshf{\X}(-f^{-1}(\Delta))\to\tshf{\Y}(-\Delta).
\end{equation}  
\qed
\end{example}

\begin{example}
  Let $f_n : \AF[1] \to \AF[1]$ be the $n$-fold branched
  cover $z \mapsto z^n$, $n \ge 1$.  Then $\R = (n-1)\cdot 0$, 
  and if $\D = q\cdot 0$ for $q\geq 1$, we have $\R \leq f_n^*\D= nq\cdot 0$,
  and an \'etale algebroid morphism 
	\begin{equation}
		\tshf{\AF[1]}(-(nq-n+1)\cdot 0)\to\tshf{\AF[1]}(-q\cdot 0),
	\end{equation}  
  covering the map $f_n$.
\qed
\end{example}

A Lie algebroid $\sA$ over $\X$ has an associated singular distribution, given by the image of the anchor map $a(\sA) \subset \tshf{\X}$. This distribution integrates to a singular 
holomorphic foliation of $\X$, whose leaves are called the
\defn{orbits} of $\sA$.  The orbits of the Lie algebroids
$\tshf{\X}(-\D)$ considered above comprise the connected
components of $\X\setminus\D$, as well as the individual points of $\D$, where the stabilizer subalgebra jumps, as we now describe.
%

\begin{definition}
  If $\sA$ is a Lie algebroid on $\X$ and $p \in \X$, 
  the \defn{isotropy Lie algebra} $\iso[p]{\sA}$ is the kernel of the restriction
  to $p$ of the anchor map, viewed as a vector bundle morphism, i.e., it is defined by the following exact sequence of vector spaces:
\begin{equation}
\xymatrix{ 0 \ar[r] & \iso[p]{\sA} \ar[r]& \A_p \ar[r]^{a|_p} &
  \tb[p]{\X}.}
\end{equation}
\end{definition}

\begin{proposition}\label{prop:iso-alg}
  Let $\sA = \tshf{\X}(-\D)$.  If $p\in \X\setminus\D$, then $\iso[p]{\sA} = \{0\}$, whereas if $p \in \D$ has multiplicity $k\geq 1$, 
then $\iso[p]{\sA}$ is an abelian Lie algebra of dimension 1, and we have a natural isomorphism
\begin{equation}
\iso[p]{\sA} \cong \rbrac{\ctb[p]{\X}}^{\otimes (k-1)}.
\end{equation}
\end{proposition}

\begin{proof}
The anchor map $a\in \cohlgy[0]{\X,\sA^\vee\otimes\tshf{\X}}$ is 
nonvanishing along $\X\setminus\D$, so $\sA$ has trivial isotropy there. 
If $p \in \D$ is a point of multiplicity $k$, then the $(k-1)$-jet of
$a$ vanishes at $p$, but the $k$-jet does not.  The exact jet
sequence
$$
0 \to \rbrac{\forms[1]{\X}}^{k}\otimes(\sA^\vee\otimes\tshf{\X}) \to
  \sJet[k]{(\sA^\vee\otimes\tshf{\X})} \to
  \sJet[k-1]{(\sA^\vee\otimes\tshf{\X})} \to 0
$$
then implies that the $k$-jet $(j^ka)|_p \in (\ctb[p]{\X})^k\otimes \A^\vee_p \otimes \tb[p]{\X}$ defines an isomorphism
$$
\A_p \cong \rbrac{ \ctb[p]{\X} }^{k-1}.
$$
Since the rank of $a$ is zero at $p$, we have $\iso[p]{\sA} = \A_p$, as required.
\end{proof}

\subsection{Lie algebroid modules and meromorphic connections}\label{sec:liealgmod}

A flat connection on a vector bundle, or on a sheaf of $\sO{\X}$-modules, may be viewed as a representation of, or module over, the Lie algebroid $\tshf{\X}$ of vector fields.  The notion of a connection carries over directly to any Lie algebroid~\cite{MR1237825, MR1325261}.  As a result, each Lie algebroid $\sA$ has an associated abelian category $\Mod{\sA}$ of $\sA$-modules, as well as a subcategory $\Rep{\sA}$ of representations, i.e., $\sA$-modules that are locally free $\sO{\X}$-modules of finite rank.

\begin{definition}\label{def:arep}
Let $\sA$ be a Lie algebroid on the complex manifold $\X$.  An \defn{$\sA$-connection} on the $\sO{\X}$-module $\sE$ is a $\CC_{\X}$-linear morphism
\begin{equation}
\nabla : \sE \to \sA^\vee \otimes_{\sO{\X}} \sE
\end{equation}
satisfying the Leibniz rule $\nabla(fs) = d_{\sA}f \otimes s + f\nabla s$ for all $f \in \sO{\X}$ and $s \in \sE$.  Here, $d_\sA f = a^\vee(df) \in \sA^\vee$ is the Lie algebroid differential of $f$.  
If
\begin{equation}
\nabla_{[\xi,\eta]}= \nabla_\xi\nabla_\eta - \nabla_\eta\nabla_\xi
\end{equation}
for all $\xi,\eta \in \sA$, then $\nabla$ is said to be flat, and $(\sE,\nabla)$ is called an \defn{$\sA$-module}. If $\sE$ is a vector bundle, then we call $(\sE,\nabla)$ a \defn{representation} of $\sA$. \qed
\end{definition}

In the case of the Lie algebroid $\sA=\tshf{\X}(-\D)$ associated to an effective divisor on a curve, we have
$\sA^\vee = \forms[1]{\X}(\D)$, so an $\sA$-connection is precisely a meromorphic connection with poles bounded by $\D$. Such a connection is always flat because $\dim\X = 1$.

\begin{proposition}
A $\tshf{\X}(-\D)$-module is an
 $\sO{\X}$-module $\sE$, equipped with a meromorphic connection $\nabla : \sE \to \forms[1]{\X}(\D)\otimes \sE$.\qed
\end{proposition}

Every Lie algebroid carries two natural representations: one 
is the trivial module $(\sO{\X}, d_\sA)$, and the other, described in~\cite{Evens1999}, is the \defn{canonical module} $(\can[\sA],\nabla)$, given by  
\begin{equation}
\can[\sA] = \det{\sA}\otimes \can[\X],
\end{equation}
where $\can[\X] = \det{\forms[1]{\X}}$ is the canonical line bundle of $\X$, and 
\begin{equation}
\nabla_\xi (\mu\otimes \nu) = \rbrac{\lie{\xi}\mu} \otimes \nu + \mu \otimes \rbrac{ \lie{a(\xi)}\nu},
\end{equation}
for $\mu \in \det{\sA}$, $\nu \in \can[\X]$ and $\xi \in \sA$.  Here $\lie{\xi}\mu$ denotes the natural action of $\sA$ on $\det{\sA}$ obtained by extending the Lie bracket on $\sA$ to its exterior algebra. 
\begin{example}
If $\D$ is an effective divisor on the curve $\X$, then the canonical $\tshf{\X}(-\D)$-module is  
\begin{equation}
\can[\X,\D] = \tshf{\X}(-\D) \otimes \forms[1]{\X} = \sO{\X}(-\D),
\end{equation}
or in other words, the ideal sheaf of $\D$ itself. If $p \in \D$ is a point of order $k$ and $z$ is a coordinate on $\X$ centred at $p$, then the section $f = z^k \in \sO{\X}(-\D)$ gives a local trivialization, and the connection is given by
\begin{equation}
\nabla f = df = kz^{k-1} dz = k z^{-1}dz \otimes f. 
\end{equation}
Hence this connection has a first-order pole at $p$, with local flat sections given by constant multiples of the meromorphic section $z^{-k}f$.\qed
\end{example}

\begin{example}\label{torsmod}
The anchor provides a morphism of $\sA$-modules $\can[\sA]\to\sO{\X}$; in the example above, this is a flat section of $\can[\X,\D]^\vee$ defining the divisor $\D$. The cokernel of this map of line bundles is 
\begin{equation}
\sO{\D} = \sO{\X}/\sO{\X}(-\D),
\end{equation}
giving a canonical example of a $\tshf{\X}(-\D)$-module that is a torsion sheaf of $\sO{\X}$-modules.\qed
\end{example}

If $(\sE,\nabla)$ is an $\sA$-module and $p \in \X$, then the connection induces a representation of the isotropy Lie algebra $\iso[p]{\sA}$ on the fibre $\sE|_p = \sE/\m_p\sE$ of the module over $p$:
\begin{equation}\label{eq:isoact}
\iso[p]{\nabla} : \iso[p]{\sA} \to \End[\CC]{\sE|_p}.
\end{equation}

\begin{proposition}\label{prop:iso-alg-act}
Let $\sA = \tshf{\X}(-\D)$, and let $\nabla$ be an $\sA$-connection on the locally free sheaf $\sE$.  Then the isotropy representation at a point $p \in \D$ of multiplicity $k$ is the element
\begin{equation}
\iso[p]{\nabla} \in (\tb[p]{\X})^{k-1} \otimes \End[\CC]{\sE|_p}
\end{equation}
determined by the leading term in the principal part of the connection at $p$.
\end{proposition}

\begin{proof}
In a local coordinate $z$ centred at $p$ and a local frame for $\sE$, we have
\[
\nabla = d + A(z)z^{-k} dz,
\]
with $A$ a holomorphic matrix-valued function.  The basis element $\delta = z^k\cvf{z} \in \sA$ acts by
\[
\nabla_\delta = z^k \cvf{z} + A(z).
\]
The isotropy action is given by the restriction of this operator to $z=0$, where we obtain the matrix $A(0)$, which is the coefficient of the leading term in the principal part of the connection.  Using the isomorphism $\iso[p]{\sA}\cong(\ctb[p]{\X})^{k-1}$ of \autoref{prop:iso-alg} we see that this matrix transforms as an element of the vector space $\Hom[\CC]{{\ctb[p]{\X}^{k-1},\End[\CC]{\sE|_p}}} = (\tb[p]{\X})^{k-1} \otimes \End[\CC]{\sE|_p}$.
\end{proof}

\begin{remark}

For $k=1$, the isotropy representation $\iso[p]{\nabla}$ is simply the residue of the regular connection $\nabla$ at $p$, a well-defined endomorphism of the fibre of $\sE|_p$. 

The weights of the isotropy action~\eqref{eq:isoact} are the eigenvalues of the principal part of the connection and give privileged elements $\lambda_1,\ldots,\lambda_r \in (\tb[p]{\X})^{k-1}$.  The corresponding weights for the adjoint action on $\End[\CC]{\sE|_p}$ are given by $q_{ij} = \lambda_j - \lambda _i$.

If $k \ge 2$, the weights $q_{ij}$ define a subset of $(\tb[p]{\X})^{k-1}$, and any tangent vector $v \in \tb[p]{\X}$ such that $v^{k-1} = q_{ij}$ is called an \defn{$(i,j)$-anti-Stokes vector}.
In a coordinate $z$ centred at $p$, we may write
\begin{equation}
\lambda_i = a_i z^{-k}dz|_p = a_i(\cvf{z})^{k-1}|_p,
\end{equation}
with $a_i \in \CC$.  Then the $(i,j)$ Stokes vectors represent the directions in which the function $\exp(({a_i-a_j})/{z^{k-1}})$ decays most rapidly as $z \to 0$.  These directions play an important role in the asymptotic analysis of the flat sections of $\nabla$, and the classification of meromorphic connections in terms of generalized monodromy data; see \cite{Boalch2001}.

A natural generalization would be to consider representations $(P,\nabla)$ of the Lie algebroid $\sA$ on a principal $G$-bundle $P$.  In this case, $\nabla$ is a lift of the algebroid morphism $a:\sA\to \tshf{\X}$ to the Atiyah algebroid $\sA(P)$ of $P$:
\begin{equation}
\begin{aligned}
\xymatrix@C=3ex@R=3ex{
 & & &\sA\ar[d]^-{a}\ar[dl]_-{\nabla} & \\
0\ar[r] &\mg\ar[r] &\sA(P)\ar[r] & \tshf{\X}\ar[r] & 0
}
\end{aligned}
\end{equation}
As a result, the isotropy representation at a point $p$ of multiplicity $k$ in $\D$ gives a Lie algebra homomorphism 
\begin{equation}
\iso[p]{\nabla} : (\ctb[p]{\X})^{k-1} \to \mg_p
\end{equation}
to the fibre of the adjoint bundle of Lie algebras over $p$.
As before, each adjoint weight is an element in $(\tb[p]{\X})^{k-1}$, whose $(k-1)$st roots define Stokes vectors in the tangent space to $p$.\qed  
\end{remark}

\vspace{1ex}
\noindent{\bf Locally free modules}
\vspace{1ex}

A module over the tangent sheaf $\tshf{\X}$ that is coherent as an $\sO{\X}$-module is necessarily locally free over $\sO{\X}$~\cite[VI, Proposition 1.7]{Borel1987}---i.e., a vector bundle.  However, other Lie algebroids may admit modules that are $\sO{\X}$-coherent but not locally free, as we saw in Example~\ref{torsmod}, where the structure sheaf of the divisor $\D$ was  shown to be a $\tshf{\X}(-\D)$-module. 

More generally, we say that a closed analytic subspace $\Y \subset \X$ is an $\sA$-subspace if its ideal sheaf $\sI_\Y \subset \sO{\X}$ is preserved by the action of $\sA$ on $\sO{\X}$.  Then the structure sheaf $\sO{\Y} = \sO{\X}/\sI_\Y$ is naturally an $\sA$-module.  Furthermore, the anchor map and bracket restrict to give $\sA|_\Y$ the structure of a Lie algebroid on $\Y$, and any $\sA$-module restricts to an $\sA|_\Y$-module.

\begin{proposition}
Let $\sA$ be a Lie algebroid on the complex manifold $\X$ and let $(\sE,\nabla)$ be a coherent sheaf equipped with an $\sA$-connection.  Then the support of $\sE$ is an $\sA$-subspace.  If, moreover, the anchor map $a : \sA \to \tshf{\X}$ is surjective, then $\sE$ is locally free, hence supported on all of $\X$.
\end{proposition}

\begin{proof}
If $f \sE = 0$ for some $f \in \sO{\X}$ and if $\xi \in \sA$, then
\[
0 = \nabla_\xi(f \sE) = \xi(f)\sE + f \nabla_\xi \sE = \xi(f) \sE.
\]
Hence $\xi(f)$ also annihilates $\sE$, and so the support is an $\sA$-subspace.

Now suppose $a$ is surjective.  Since the question is local in nature, we may assume that there is an $\sO{\X}$-linear splitting $b : \tshf{\X} \to \sA$ with $ab = \id{\tshf{\X}}$.  Then $b^\vee \circ \nabla : \sE \to \forms[1]{\X} \otimes \sE$ is a connection on $\sE$ in the usual sense.  This connection need not be flat, and hence it does not, in general, give $\sE$ the structure of a module over the sheaf $\sD_{\X}$ of differential operators.  Nevertheless, the proof in \cite[VI, Proposition 1.7]{Borel1987} may still be applied using this connection to show that $\sE$ is locally free.
\end{proof}

\begin{corollary}
If $\Y \subset \X$ is an orbit of $\sA$, then the restriction of any $\sA$-module to $\Y$ is a locally free $\sO{\Y}$-module.
\end{corollary}

For example, over a curve $\X$ with an effective divisor $\D$, if $(\sE,\nabla)$ is a $\tshf{\X}(-\D)$-module, then $\sE|_{\X\setminus \D}$ is locally-free.  The rank of $\sE$ can only jump at the points of $\D$.

\subsection{Functorial operations on modules}\label{funcmod}

Meromorphic connections may be pulled back and pushed forward along morphisms of curves; in this section, we use the behaviour of Lie algebroids under morphisms described in~\textsection\ref{sec:liealgcur} to emphasize the control one has on the singularities of connections obtained by these natural operations.

Let $f:\X\to\Y$ be a morphism between curves and $(\sE,\nabla)$ an $\sA_\Y$-module, for the algebroid $\sA_\Y = \tshf{\Y}(-\D)$ defined by an effective divisor $\D$.  Using the morphism $f^!\sA_\Y\to f^*\sA_\Y$, the pullback $f^*\nabla$ defines a $f^!\sA_\Y$-module structure on $f^*\sE$. 
By~\autoref{prop:pullie}, we have $f^!\tshf{\Y}(-\D) = \tshf{\X}(-f^*\D\setminus\R)$, so we have a canonical pullback functor
\begin{equation}
f^*: \Mod{\tshf{\Y}(-\D)} \longrightarrow \Mod{\tshf{\X}(-f^*\D\setminus\R)},
\end{equation}
implying that the pullback of a connection with poles bounded by $\D$ has poles bounded by $f^*\D\setminus\R$. 

The pushforward of $\sO{\X}$-modules can be used to map $\sA_\X$-modules to $\sA_\Y$-modules when these algebroids are related by an \'etale morphism in the sense of~\autoref{def:etale}:

\begin{proposition}\label{prop:dirim}
Let $f:\X\to\Y$ be a morphism of curves equipped with effective divisors $\C$ and $\D$, respectively, such that $\tshf{\X}(-\C)\to\tshf{\Y}(-\D)$ is \'etale, i.e., $\R\leq f^*\D$ and $\C = f^*\D - \R$.  Then the direct image $f_*$ extends to a functor
\begin{equation}
f_*: \Mod{\tshf{\X}(-\C)} \longrightarrow \Mod{\tshf{\Y}(-\D)}.
\end{equation}
\end{proposition}

\begin{proof}
Let $\sA_\X = \tshf{\X}(-\C)$ and $\sA_\Y=\tshf{\Y}(-\D)$, and let $(\sE,\nabla)\in \Mod{\sA_\X}$, so that $\nabla:\sE\to \sA_\X^\vee\otimes\sE$. Taking the direct image, we see that $f_*\nabla$ maps $f_*\sE$ to $f_*(\sA_\X^\vee\otimes \sE)$. But the morphism between $\sA_\X$ and $\sA_\Y$ is \'etale, implying that $\sA_\X = f^!\sA_\Y = f^*\sA_\Y$, giving
$$
f_*\nabla: f_*\sE\longrightarrow \sA_\Y^\vee\otimes f_*\sE.
$$
The map $f_*\nabla$ also satisfies the Leibniz rule, since the pullback of the de Rham operator $d_\Y:\sO{\Y}\to \sA_{\Y}^\vee$ coincides with the de Rham operator of $\sA_\X$.
\end{proof}
\begin{remark}
Recall that the pushforward $f_*\sE$ of an $\sO{\X}$-module carries an action of the sheaf of rings $\sS = f_*\sO{\X}$.  In the above situation of an \'etale algebroid morphism, the de Rham operator $d_\X:\sO{\X}\to \sA_\X^\vee$ has direct image 
\begin{equation*}
\sD : \sS \to \sA_\Y^\vee\otimes \sS,
\end{equation*}  
and satisfies an obvious Leibniz rule.  Furthermore, the direct image of an $\sA_\X$-module via~\autoref{prop:dirim} is characterized by the fact that it satisfies an extension of the usual Leibniz rule over $\sO{\Y}$: 
for all $h\in\sS$ and $s \in f_*\sE$, we have 
\[
\pushQED{\qed}
(f_*\nabla)(hs) = h\cdot ((f_*\nabla)s) + \sD(h)\cdot s.\qedhere
\popQED
\]
\end{remark}

\begin{example}
For any morphism of curves $f:\X\to\Y$, choose effective divisors as in~\autoref{ex:branchcover}, so that $\D$ is the reduced divisor of critical values of $f$ and $\C$ is the preimage of $\D$ in $\X$, defining an \'etale algebroid morphism.

The pushforward $f_*\sO{\X}$ of the trivial module is a vector bundle over $\Y$ of rank $d = \deg f$, and it is equipped with a meromorphic connection with simple poles along $\D$. In fact, the monodromy of this connection along $\Y\setminus\D$ is a permutation representation into $S_d$; it is precisely the representation which defines the branched covering $f:\X\to\Y$ itself.  

The other natural module over $\X$ is the canonical $\tshf{\X}(-\C)$-module, namely $\sO{\X}(-\C)$.  Its pushforward is again a vector bundle over $\Y$ of rank $d=\deg f$, equipped with a connection with simple poles along $\D$. As we saw in~\autoref{torsmod}, $\sO{\X}(-\C)$ admits a flat trivialization away from $\C$, and so the pushforwards of $\sO{\X}$ and $\sO{\X}(-\C)$ are isomorphic away from $\D$.

For a general line bundle $\sL\in\Mod{\tshf{\X}(-\C)}$, the pushforward is a rank $d$ bundle whose connection has simple poles along $\D$, but because of the possible lack of global flat sections of $\sL$ along $\X\setminus\C$, the monodromy of $f_*\sL$ along $\Y\setminus \D$ can only be said to be a quasi-permutation representation; see~\cite{MR2060367}.\qed
\end{example}

\begin{example}
Let $\X = \Y = \AF[1]$ be the affine complex line, and let $f:\X\to\Y$ be the morphism $z\mapsto w = z^2$, which lifts to an \'etale morphism from $\tshf{\X}(-3\cdot 0)$ to $\tshf{\Y}(-2\cdot 0)$.   A general $\tshf{\X}(-3\cdot 0)$-module of rank 1 has the form 
\[
\nabla = d + A z^{-3}dz,
\]
for $A\in\sO{\X}$. To compute the pushforward connection $f_*\nabla$ with respect to the $\sO{\Y}$-basis $(1, z)$, we write 
\[
\nabla 1 = A z^{-3}dz = \tfrac{1}{2} (A_0 + z A_1) w^{-2} dw,
\] 
for $A_0,A_1\in\sO{\Y}$, and similarly
\[
\nabla z = dz +  Az^{-2}dz = 
\tfrac{1}{2}(A_1w^{-1}dw + 
z (A_0 w^{-2} + w^{-1})dw) .
\]
Therefore, the pushforward connection can be written 
\[
\pushQED{\qed}
f_*\nabla = d + 
\frac{1}{2}\begin{pmatrix}A_0&w A_1\\A_1&A_0 + w\end{pmatrix}w^{-2}dw .\qedhere
\popQED
\]
\end{example}

\subsection{Differential operators of higher order}
\label{sec:higher-order-ops}

Lie algebroid connections are first-order differential operators: the action of an element of $\sA$ on a section $s \in \sE$ depends only on the one-jet of $s$.  In order to model differential operators of higher order, it is useful to recall the notion of jets for Lie algebroids as in \cite{Calaque2010}.  Associated to any Lie algebroid $\sA$ over $\X$ is a universal envelope $\sD_{\sA}$, which is the quotient of the $\sO{\X}$-tensor algebra of $\sA$ by the $\CC_\X$-linear relations
\begin{equation*}
\begin{aligned}
f \otimes \xi &= f\xi \\
\xi \otimes f - f \otimes \xi &= \lie{a(\xi)}f \\
\xi \otimes \eta - \eta \otimes \xi &= [\xi,\eta]
\end{aligned}
\end{equation*}
for all $f \in \sO{\X}$ and $\xi,\eta \in \sA$.  When $\sA = \tshf{\X}$, this construction yields the usual sheaf $\sD_\X$ of differential operators on $\X$.  The sheaf $\sD_\sA$ has a natural filtration by the \defn{order} of a differential operator: $\sD_\sA^{\le k}$ is simply the $\sO{\X}$-submodule of $\sD_\sA$ generated by the tensors of degree $\le k$. 

\begin{definition}
If $\sE$ is an $\sO{\X}$-module, then the \defn{sheaf of $k$-jets} of $\sE$ along $\sA$ is the $\sO{\X}$-module
$$
\xymatrix{
\sJet[k][\sA]{\sE} = \sHom[\sO{\X}]{\sD_\sA^{\le k}, \sE}.
}
$$
\end{definition}

When $\sA = \tshf{\X}$, we recover the usual sheaf of $k$-jets of $\sE$, which we denote by $\sJet[k]{\sE}$.  The anchor map $a : \sA \to \tshf{\X}$ induces a morphism
$$
a^*_k : \sJet[k]{\sE} \to \sJet[k][\sA]{\sE},
$$
and so given a section $s \in \sE$, we may define its $k$-jet along $\sA$ as
$$
j^k_\sA s = a_k^*(j^k s),
$$
where $j^k s$ is the $k$-jet of $s$ in the usual sense.

The usual jet sequence for an $\sO{\X}$-module $\sE$ may be pushed out along the dual anchor $\sym[k]a^*:\sym[k]{\tshf{X}^\vee}\to \sym[k]{\sA^\vee}$, yielding an analogous exact sequence for algebroid jets:
$$
\xymatrix{
0 \ar[r] & \sym[k]{\sA^\vee} \otimes \sE \ar[r] & \sJet[k][\sA]{\sE} \ar[r] & \sJet[k-1][\sA]{\sE} \ar[r] & 0.
}
$$

\begin{definition}
A  \defn{$k$th order $\sA$-connection on $\sE$} is a splitting
$$
\phi : \sJet[k][\sA]{\sE} \to \sym[k]{\sA^\vee}\otimes \sE
$$
of the $k$-jet sequence.  Equivalently, $Ds = \phi(j^{k}_\sA s)$ defines a linear differential operator of order $k$ from $\sE$ to $\sym[k]{\sA^\vee}\otimes \sE$ whose symbol is the identity. 
\end{definition}

Consider the situation for the Lie algebroid $\tshf{\X}(-\D)$ associated to an effective divisor on a curve.  If $p \in \D$ is a point of order $n$ and we choose a coordinate $z$ centred at $p$, then $\sA = \tshf{\X}(-\D)$ is generated locally by the vector field $\delta = z^n\cvf{z}$.  Let $\alpha = z^{-n}dz \in \sA^\vee$ be the dual basis.   The elements $1,\delta,\delta^2,\ldots,\delta^k$ form a basis for $\sD_\sA^{\le k}$.  With respect to the dual basis $1,\alpha,\ldots,\alpha^k$, the jet of a section $f \in \sO{\X}$ is given by
$$
j^k_\sA f = f \cdot 1 + \delta(f) \alpha + \cdots + \delta^k(f) \alpha^k.
$$
A general $k$th order $\sA$-connection on $\sO{\X}$ therefore locally has the form
\begin{equation}\label{expho}
D(f) = \phi(j^k_{\sA}f) = \rbrac{\delta^k(f) + p_{k-1}\delta^{k-1}(f) + \cdots + p_0f} \alpha^k,
\end{equation}
for holomorphic functions $p_0,\ldots,p_{k-1}$. For $k=2$, we have the familiar form
\begin{align}
D(f)= \rbrac{z^{2n} \deriv{^2\!f}{z^2} + (p_1 + nz^{n-1})z^n\deriv{f}{z} + p_0 f}\alpha^2.\label{eqn:2nd-order-op}
\end{align}

\begin{example}[\airylabel]\label{ex:airy-2nd-order}
An important early example in the study of differential equations with irregular singularities was Stokes' analysis~\cite{Stokes1847} of the Airy equation.  Recall that the Airy equation is the second-order differential equation
$$
\deriv{^2\!f}{x^2} = x f.
$$
This corresponds to the second order connection (for $\sA = \tshf{\AF[1]}$) on $\sO{\AF[1]}$ given by
$$
D(f) = \rbrac{\deriv{^2\!f}{x^2} - xf}dx^2.
$$
In the coordinate $z = x^{-1}$ on $\PP[1]$, we have
\begin{align*}
D(f) &= \rbrac{z^4 \deriv{^2\!f}{z^2} +2z^3 \deriv{f}{z}- z^{-1}f}\rbrac{\frac{dz}{z^2}}^2 \\
&= \rbrac{ z^6\deriv{^2\!f}{z^2} +(2z^2)z^3 \deriv{f}{z}- zf} \rbrac{\frac{dz}{z^3}}^2 \\
&= \rbrac{\delta^2(f) - z^2\delta(f) - zf}\rbrac{\frac{dz}{z^3}}^2,
\end{align*}
where $\delta = z^3\cvf{z}$.  Therefore $D$ extends to give a second order connection 
$$
D_{\mathrm{Airy}} : \sO{\PP[1]} \to 
\sym[2]{\sA^\vee}\otimes \sO{\PP[1]}
$$
for the Lie algebroid $\sA=\tshf{\PP[1]}(-3 \cdot \infty)$ on $\PP[1]$.\qed
\end{example}

In studying a differential equation of order $k$ for a function $f$, it is often useful to convert the equation into a system of first-order equations for $f$ and its first $k-1$ derivatives.  In general, a $k$th order connection on $\sE$ may be expressed as a usual first order connection on $\sJet[k-1][\sA]\sE$. We briefly describe this transformation, following~\cite{MR2595844}.

The natural embedding
$$
\sJet[k][\sA]{\sE} \hookrightarrow\sJet[1][\sA]{\sJet[k-1][\sA]{\sE}}
$$
fits into a commutative diagram
$$
\xymatrix@C=4ex@R=4ex{
0 \ar[r] & \sym[k]{\sA^\vee} \otimes \sE \ar[r] \ar[d] & \sJet[k][\sA]{\sE} \ar[r]\ar[d] & \sJet[k-1][\sA]{\sE} \ar[r]\ar@{=}[d] & 0 \\
0 \ar[r] & \sA^\vee \otimes \sJet[k-1][\sA]{\sE}  \ar[r]  & \sJet[1][\sA]{\sJet[k-1][\sA]{\sE}} \ar[r] & \sJet[k-1][\sA]{\sE} \ar[r] & 0 \\
}
$$
As a result, any splitting of the $k$-jet sequence for $\sE$ induces a splitting of the $1$-jet sequence for $\sJet[k-1][\sA]{\sE}$, i.e., a connection on $\sJet[k-1][\sA]{\sE}$.  On the other hand, if a usual connection 
\[
\nabla:\sJet[k-1][\sA]{\sE}\to \sA^\vee\otimes \sJet[k-1][\sA]{\sE}
\] 
comes from a $k$th order connection on $\sE$, then its composition with the projection $\pi: \sJet[k-1][\sA]{\sE}\to \sJet[k-2][\sA]{\sE}$ coincides with the Spencer operator 
\[
\mathcal{S}:\sJet[k-1][\sA]{\sE}\to \sA^\vee\otimes \sJet[k-2][\sA]{\sE}.  
\]
Hence we obtain a necessary condition for a connection on $\sJet[k-1][\sA]{\sE}$ to define a $k^{\text{th}}$ order connection on $\sE$:
\begin{equation}\label{spence}
(\id{\sA^\vee}\otimes \pi) \circ \nabla = \mathcal{S}.
\end{equation}
Apart from a constraint on the curvature of $\nabla$, which is vacuous in the 1-dimensional case, this condition is also sufficient, as shown in \cite[Theorem 2]{MR2595844}.
\begin{proposition}\label{prop:ho-conn}
Let $\sA$ be a Lie algebroid on the curve $\X$ and let $\sE$ be an $\sO{\X}$-module.  Then there is a bijective correspondence between $k$\emph{th} order $\sA$-connections on $\sE$ and $\sA$-connections $\nabla$ on $\sJet[k-1][\sA]{\sE}$ which satisfy condition~\eqref{spence}.
\end{proposition}

Returning to the explicit form~\eqref{expho} of the general $k$th order $\sA$-connection on $\sO{\X}$, we see that in the basis $1,\alpha,\cdots,\alpha^{k-1}$ for $\sJet[k-1][\sA]{\sO{\X}}$, the corresponding connection is given by
\begin{align}
  \nabla = d + 
  \begin{bsmallmatrix*}[r]
    0 & -1 & 0 & \cdots & 0 &\phantom{\scriptsize{k-1}} \\
    0 & 0 & -1 & \cdots & 0 & \\
    \vdots & \vdots & \vdots & \ddots & \vdots & \\
    0 & 0 & 0 & \cdots & -1 & \\
    p_{\mathrlap{0}} & p_{\mathrlap{1}} & p_{\mathrlap{2}} & \cdots & p_{\mathrlap{{k-1}}} & 
  \end{bsmallmatrix*} z^{-n}dz.\label{eqn:jet-conn}
\end{align}

\begin{example}[\airylabel]\label{ex:airy}
Continuing from \autoref{ex:airy-2nd-order}, where $\X = \PP[1]$ and $\sA = \tshf{\PP[1]}(-3\cdot \infty) \cong \sO{\PP[1]}(-1)$, the second order Airy operator $D_{\mathrm{Airy}}$ defines a usual $\sA$-connection on the rank-two jet bundle 
$$
\sE = \sJet[1][\sA]{\sO{\PP[1]}} \cong \sO{\PP[1]} \oplus \sO{\PP[1]}(1).
$$
In the coordinate $x$ and the corresponding basis $1, dx$ for $\sE$, this connection is 
$$
\nabla^{\Airy} = d + \begin{pmatrix}
0 & -1 \\
-x & 0
\end{pmatrix}
dx,
$$
while in the coordinate $z$ at infinity, with basis $1,z^{-3}dz$ for $\sE$, the connection is 
$$
\nabla^\Airy =  d + \begin{pmatrix}
0 & -1 \\
-z & -z^2
\end{pmatrix}z^{-3}dz,
$$
verifying that $\nabla^\Airy$ is a representation of $\tshf{\PP[1]}(-3 \cdot \infty)$ on $\sO{\PP[1]} \oplus \sO{\PP[1]}(1)$.
\qed
\end{example}



\subsection{Lie algebroid connections on general fibrations}\label{sec:ehresmancon}

We have focused so far on linear connections: actions of a Lie algebroid on sheaves of $\sO{\X}$-modules.  In order to model non-linear differential equations, we must pass to other types of fibrations.  Such a theory has been developed in the smooth category by Fernandes~\cite{Fernandes2002} and generalizes immediately to the holomorphic setting.  In this section, we briefly outline those aspects we require.

Let $\pi : \P \to \X$ be a holomorphic submersion of complex manifolds, and let $\sA$ be a Lie algebroid on $\X$. As in~\textsection \ref{sec:liealgcur}, the pullback algebroid $\pi^!\sA$ is related to the relative tangent sequence as shown in the commutative diagram 
$$
\xymatrix@R=4ex@C=4ex{
0 \ar[r] & \tshf{\P/\X} \ar[r]\ar@{=}[d] & \pi^!\sA \ar[r]^{\pi_\sA}\ar[d] & \pi^*\sA \ar[r]\ar[d]\ar@{.>}@/^/[l]^\Theta & 0\\
0 \ar[r] & \tshf{\P/\X} \ar[r] & \tshf{\P} \ar[r]_{T\pi} & \pi^*\tshf{\X} \ar[r] & 0,
}
$$
where $\tshf{\P/\X}$ is the relative tangent sheaf (the vertical bundle to $\pi$).  An \defn{Ehresmann $\sA$-connection on $\pi$} is then simply a splitting $\Theta$ of $\pi_\sA$.  If $\sA = \tshf{\X}$, then $\pi^!\sA = \tshf{\P}$ and the connection is a choice of complement to $\tshf{\P/\X}$.

If $\P$ is a fibre bundle with structure group $\G$, then $\Theta$ may be required to be $\G$-equivariant, in which case we call it an algebroid \defn{$\G$-connection}. Of course, if $\G = \GL{n,\CC}$ and $\P\to\X$ is a vector bundle, an algebroid  $\G$-connection is the same as an $\sA$-module structure on $\P$.

The connection $\Theta$ composed with the anchor of $\pi^!\sA$ defines a morphism $h:\pi^*\sA\to \tshf{\P}$ whose image is called the \defn{horizontal distribution}, by analogy with the case when $\sA = \tshf{\X}$.  Notice, though, that $\tshf{\P/\X} \cap \image{h} \ne 0$ in general.  Nevertheless,  every section $\xi \in \sA$ has a horizontal lift $h(\pi^*\xi) \in \tshf{\P}$.
As for usual Ehresmann connections, $\Theta$ has a curvature tensor $F_\Theta$ in $\wedge^2 (\pi^*\sA)^\vee\otimes \tshf{\P/\X}$; if this tensor vanishes, then we say that the connection is \defn{flat}, and in this case the horizontal distribution integrates to a possibly singular holomorphic foliation called the \defn{horizontal foliation}.

For the Lie algebroid $\sA = \tshf{\X}(-\D)$ associated to an effective divisor on a curve, the pullback $\pi^!\sA$, defined as the fiber product of $\tshf{\P}$ with $\pi^*\sA$ over $\pi^*\tshf{\X}$, can be expressed as
\[
\pi^! \sA = \tshf{\P}(- \log \pi^*\D), 
\]
where 
$\tshf{\P}(- \log \pi^*\D) = \ker{\tshf{\P} \to \pi^* \tshf{\X}/\tshf{\X}(-\D)}
$
is the sheaf of vector fields on $\P$ which are tangent to the fibres of $\pi$ over $\D$ to the prescribed order.  If $(x,y_1,\ldots,y_d)$ are coordinates on $\P$ such that $\pi$ is the projection on the first component and $x$ is a coordinate on $\X$ centred at a point $p \in \D$ of multiplicity $k$, then the vector fields $x^k\cvf{x},\cvf{y_1},\ldots,\cvf{y_d}$ form a basis of sections of $\tshf{\P}(- \log \pi^*\D)$ in this chart.  In particular, when $k=1$ it is the usual sheaf of logarithmic vector fields, as suggested by the notation.  Therefore, the horizontal leaves of any $\sA$-connection, while they are generically transverse to the fibres of $\pi$, must be \emph{tangent} to the fibers over $\D$.


\begin{example}
If $\sE \in \Rep{\sA}$ is a representation of $\sA$ of rank $r$, then the projective bundle $\PP(\sE) \to \X$ inherits a natural algebroid $\PGL{r,\CC}$-connection whose horizontal leaves are the projections of the horizontal leaves on the total space of ${\sE}$. \qed
\end{example}

\subsection{Meromorphic projective structures}\label{sec:meroproj}

The uniformization theorem for Riemann surfaces identifies the universal cover $\tilde\X$ of a Riemann surface $\X$ with either $\PP[1]$, $\AF[1]$, or the upper half plane $\HH$. The embedding  $\tilde \X\subset \PP[1]$ may be used to endow $\X$ with a covering by charts $(U_i,\varphi_i:U_i\to\PP[1])$ with the property that the gluing maps $g_{ij}=\varphi_j\varphi_i^{-1}$ are constant elements of $\PGL{2,\CC}$.
We may then view this data as defining a
$\PP[1]$-bundle $\P\to\X$ equipped with a flat connection and a section $\varphi:\X\to \P$ that, far from being flat, defines an isomorphism $\varphi^*\tshf{\P/\X}\cong \tshf{\X}$.
In this way, the problem of finding the uniformization of $\X$ may be understood as a search for an appropriate flat projective bundle with section; this observation was at the heart of one of the first approaches to uniformization by Fuchs and Poincar\'e.  

Among other things, Gunning explained in \cite{MR0207978} that the existence of the section $\varphi$ implies that $\P$ may be expressed as the projectivization of a rank 2 bundle $\sE$ which, if chosen correctly, has a flat connection whose projectivization recovers the given flat structure on $\P$. The bundle $\sE$ is chosen to be the first jet bundle of any square root of the anti-canonical line bundle $\omega_\X^{-1/2}$; the jet sequence
\begin{equation}
\xymatrix{0\ar[r] &\tshf{\X}^\vee\otimes\omega_\X^{-1/2}\ar[r] & \sJet[1]{\omega_\X^{-1/2}}\ar[r] & \omega_\X^{-1/2}\ar[r] & 0}
\end{equation}
defines the section $\varphi:\X\to \P=\PP{\sE}$, which then has the property $\varphi^*\tshf{\P/\X}\cong \tshf{\X}$. Finally, the flat connection on $\sE=\sJet[1]{\omega_\X^{-1/2}}$ actually derives from a certain second order connection (a Cartan projective connection) on the line bundle $\omega_\X^{-1/2}$, in the manner described by~\autoref{prop:ho-conn}.  In conclusion, we see that in this approach, the uniformization of $\X$ may be described in terms of a very particular second order ordinary differential equation on the sections of a square root of the canonical line bundle. 
 
The theory of projective structures outlined above has an analog for curves $\X$ equipped with effective divisors $\D$, in which the tangent sheaf is replaced with the Lie algebroid $\sA = \tshf{\X}(-\D)$.   
\begin{definition}\label{def:mer-ps}
Let $\omega_\X^{-1/2}$ be a square root of the anti-canonical bundle of the curve $\X$, and let $\D$ be an effective divisor, with associated Lie algebroid $\sA= \tshf{\X}(-\D)$.  A \defn{projective $\sA$-connection} on $(\X,\D)$ is a second order $\sA$-connection
$$
D : {\can[\X]^{{-1}/{2}}} \to \sym[2]{\sA^\vee} \otimes \can[\X]^{{-1}/{2}},
$$
such that the $\sA$-connection on the jet bundle of $\omega_\X^{-1/2}$ induced via~\autoref{prop:ho-conn}
has determinant equal to the dual of the canonical module given in~\textsection\ref{sec:liealgmod}, i.e., 
$$
\det \sJet[1][\sA]{\can[\X]^{{-1}/{2}}} \cong \sO{\X}(\D).
$$
as $\sA$-modules.
\end{definition}
Such a projective $\sA$-connection induces an $\sA$-connection on the rank-two bundle $\sJet[1][\sA]{\can[\X]^{-1/2}}$, and thus an algebroid $\PGL{2,\CC}$-connection on the associated projective line bundle $\P\to\X$, as defined in \textsection\ref{sec:ehresmancon}.  The algebroid jet sequence then provides a section $\varphi:\X\to\P$ such that $\varphi^*\tshf{\P/\X}$ is identified with $\sA$ by the Ehresmann $\sA$-connection.    In \textsection\ref{sec:unif}, we explain how, instead of describing the uniformization of $\X$, such a meromorphic projective structure may be used to describe the Lie groupoid integrating $\sA$.

For an explicit description of the projective $\sA$-connection, we use equations \eqref{eqn:2nd-order-op} and \eqref{eqn:jet-conn} to show that any projective $\sA$-connection $D$ has the local form
\begin{equation}\label{conndef}
D(f dz^{-1/2})  = \rbrac{f'' +  q_{z}f} dz^{3/2},
\end{equation}
where $q_{z} = z^{-2k}p_0$, for some $p_0 \in \sO{\X}$.  We may call $q_{z}$ the connection coefficient of $D$ in the coordinate $z$; if $w = g(z)$ is another coordinate, then
\begin{equation}\label{projchan}
q_{w}\cdot (g')^2= q_{z} + S(g),
\end{equation}
where
$$
S(g) = {\frac{g'''}{g'} - \frac{3}{2}\rbrac{\frac{g''}{g'}}^2}
$$
is the Schwarzian derivative of $g$.
%
%


%% file: stokes1-constr.tex

Just as a Lie algebra is the tangent space to a Lie group at the identity element, a Lie algebroid $\sA$ over $\X$ is the normal bundle to the submanifold $\gid(\X)\subset\G$ of identity elements in a Lie groupoid over $\X$.  Any groupoid $\G$ with Lie algebroid $\sA$ is said to \emph{integrate} $\sA$.  Our main objective is to construct Lie groupoids $\G$ over curves that integrate the algebroids $\sA = \tshf{\X}(-\D)$ described above.

\subsection{Review of Lie groupoids}

A \defn{(holomorphic) Lie groupoid} is a tuple $(\G,\X,s,t,m,\gid)$ defining a groupoid whose arrows and objects form complex manifolds\footnote{Sometimes $\G$ is assumed to satisfy all manifold axioms except the Hausdorff condition; we shall see one such example in~\autoref{sec:unif}.} $\G$ and $\X$, respectively.  The maps $s,t : \G \to \X$ taking an arrow to its source and target are required to be holomorphic submersions, and the composition of arrows is a holomorphic map
$$
m : \G {_s\times_t} \G \to \G.
$$
The map $\gid : \X \to \G$ is a closed embedding taking $x \in \X$ to the identity arrow of $x$.  If $g,h \in \G$ satisfy $s(g) = t(h)$, we write $gh = m(g,h)$ for their composition.
A morphism between Lie groupoids $\G$ and $\H$ is then a holomorphic map $F : \G \to \H$ on arrows that covers a holomorphic map on objects and defines a functor.

The Lie algebroid $\LA[\G]$ associated to a Lie groupoid $\G$ is, as a vector bundle, given by the normal bundle to the embedding of $\X$ in $\G$ by the identity:
\[
\LA[\G] = \sN_{\G}\X.
\]
The Lie bracket on the sections of $\LA[\G]$ is obtained, as for Lie algebras, by identifying its sections with right-invariant vector fields on $\G$ tangent to the source fibres. Finally, the anchor map $a:\LA[\G]\to\tshf{\X}$ is induced by the derivative of the target map along $\gid(\X)$.  This construction is functorial: the derivative of any morphism of Lie groupoids $F : \G \to \H$ induces a morphism $\LA[\G] \to \LA[\H]$.

\begin{definition}
Let $\G$ be a Lie groupoid over $\X$. A \defn{$\G$-equivariant sheaf} is a pair $(\sE,\Psi)$, where $\sE$ is an $\sO{\X}$-module,  and $\Psi \in \Hom[\G]{s^*\sE, t^*\sE}$ is an isomorphism that is multiplicative, in the sense that 
$$
p_1^*\Psi \circ p_2^*\Psi = m^* \Psi,
$$
for $p_1, p_2$ the natural projections $\G {_s\times_t} \G \to \G$ and $m$ the groupoid composition.  If $\sE$ is locally free, we say that the pair $(\sE,\Psi)$ is a \defn{representation of $\G$}.

A morphism between two equivariant sheaves $(\sE,\Psi)$ and $(\sE',\Psi')$ is an $\sO{\X}$-linear map $\psi : \sE \to \sE'$ such that
$$
t^*\psi \circ \Psi = \Psi \circ s^*\psi : s^*\sE \to t^*\sE.
$$
We denote by $\Rep{\G}$ the category of representations of $\G$.\qed
\end{definition}


\begin{example}\label{ex:pair}
The \defn{pair groupoid $\Pair{\X}$} over a complex manifold $\X$ has arrows $\Pair{\X} = \X \times \X$, maps $s,t$ given by the natural projections, and identity map given by the diagonal embedding.  The composition law is 
$$
(x,y)\cdot(y,z) = (x,z),
$$
and is illustrated in~\autoref{fig:pairgpd}.  The Lie algebroid of $\Pair{\X}$ is then naturally identified with the tangent sheaf of $\X$.  An equivariant sheaf for the groupoid $\Pair{\X}$ is nothing but a constant sheaf on $\X$.
\begin{figure}
\centering
\tikzstyle{point}=[circle, fill=black, inner sep=0pt, minimum width=3pt]
\begin{tikzpicture}[scale=1.8]
\filldraw[fill=black!5!white] (0,-1)--(1,0)--(0,1)--(-1,0)--(0,-1);
\draw[very thick] (-1,0) -- (1,0);
\draw[dotted] (-0.5,0) -- (0,0.5);
\draw[dotted] (0.5,0) -- (0,0.5);
\draw[dotted] (-0.25,0.25) -- (0,0);
\draw[dotted] (0.25,0.25) -- (0,0);
\node[point] at (0.25,0.25) [label=90:$h$] {};
\node[point] at (-0.25,0.25) [label=90:$g$] {};
\node[point] at (0,.5) [label=90:$gh$] {};
\draw[->] (0.4,0.8) -- (0.8,0.4) node [above right, midway] {$s$};
\draw[->] (-0.4,0.8) -- (-0.8,0.4) node [above left, midway] {$t$};
\draw (0,-0.15) node {$\gid(\X)$};
\draw (0.6,-0.6) node {$\X$};
\draw (-0.6,-0.6) node {$\X$};
\end{tikzpicture}
\caption{The pair groupoid $\Pair{\X}$}\label{fig:pairgpd}
\end{figure}
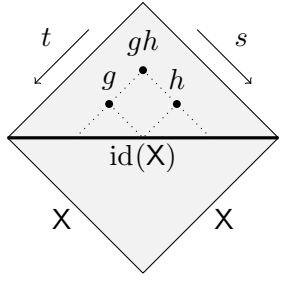
\noindent Note that every Lie groupoid $\G$ over $\X$ maps to $\Pair{\X}$ via its source and target maps; indeed $\Pair{\X}$ is terminal in the category of Lie groupoids over $\X$.\qed
\end{example}

\begin{example}\label{ex:action}
Let $\rho:\H\times\X\to \X$ be the holomorphic action of a complex Lie group $\H$ on the complex manifold $\X$. The \defn{action groupoid} $\H \ltimes \X$ over $\X$ has arrows given by $\H \times \X$, with source and target maps to $\X$ given by the natural projection and $\rho$, respectively.  
The composition is given by 
$$
(h,x)\cdot(k,y) = (hk,y),
$$
which is defined whenever $x = \rho(k,y)$, and has identity elements given by $\X \cong \{1\} \times \X \subset \H \times \X$.  The corresponding Lie algebroid  is the trivial bundle $\LA[\H]\otimes_\CC \sO{\X}$, which inherits a Lie bracket on sections from the Lie algebra $\LA[\H]$, as well as an anchor map from the derivative of $\rho$. 
Equivariant sheaves for the groupoid $\H\ltimes\X$ are nothing but $\H$-equivariant sheaves in the usual sense.
\qed
\end{example}
%
%
%
%

\begin{example}
The \defn{gauge groupoid} $\GL{\sE}$ of a locally free sheaf $\sE$ on $\X$ is the Lie groupoid over $\X$ for which the morphisms between $p,q\in\X$ are the $\CC$-linear isomorphisms $\sE|_{p} \to \sE|_{q}$ between the fibres of $\sE$.  The isotropy group at $p$ is therefore identified with $\GL{\sE|_p}$.
The Lie algebroid of $\GL{\sE}$ is the sheaf $\ati{\sE}$ of first order differential operators on $\sE$ with scalar symbols; it is called the Atiyah algebroid of $\sE$ and fits in the exact sequence
$$
\xymatrix{
0 \ar[r] & \sEnd{\sE} \ar[r] & \ati{\sE} \ar[r] & \tshf{\X} \ar[r] & 0.
}
$$
This example is important to us because a representation of a Lie groupoid $\G$ on a locally free sheaf $\sE$ is equivalent to a morphism $\Psi : \G \to \GL{\sE}$ of Lie groupoids.  The infinitesimal version of this morphism is a Lie algebroid morphism $\Psi_*:\LA[\G]\to \ati{\sE}$, i.e., a representation of $\LA[\G]$ on $\sE$ in the sense of~\autoref{def:arep}.  Therefore, the Riemann--Hilbert correspondence between representations of $\LA[\G]$ and representations of $\G$ may be viewed as an integration problem for Lie algebroid morphisms to Lie groupoid morphisms.  
 \qed
\end{example}

\subsection{Integration and source-simply connected groupoids}
\label{sec:ssc-integ}

An \defn{integration} of the Lie algebroid $\sA$ is a pair $(\G,\phi)$, where $\G$ is a Lie groupoid and $\phi:\LA[\G] \to \sA$ is an isomorphism of Lie algebroids.  Unlike finite-dimensional Lie algebras, which are always integrable to Lie groups by Lie's third theorem, a Lie algebroid may fail to integrate to a Lie groupoid.  The obstruction theory controlling this problem was described by Crainic and Fernandes in~\cite{Crainic2003a}. 
One situation in which a certain type of integration was proven to exist is when the anchor map is an injection on sections:
\begin{theorem}[Debord~\cite{Debord2001}]\label{thm:debord}
If $\sA$ is a Lie algebroid on $\X$ for which the anchor map $a : \sA \to \tshf{\X}$ is an embedding of sheaves, then $\sA$ is integrable to a Lie groupoid $\G $ having the property that the only map $\X \to \G$ that is a section of both $s$ and $t$ is the identity map $\gid : \X \to \G$.
\end{theorem}
\begin{corollary}\label{cor:tdint}
The Lie algebroid $\tshf{\X}(-\D)$ associated to a curve $\X$ with effective divisor $\D$ is integrable. 
\end{corollary}
In~\autoref{biratgpd}, we provide an explicit birational construction of the groupoid $\Pair{\X,\D}$ integrating $\tshf{\X}(-\D)$ that features in~\autoref{thm:debord}.
In fact, this groupoid is the ``smallest'' integration of $\tshf{\X}(-\D)$, in the sense that it is terminal in the category of integrations.  We are also interested in the other extreme, as follows.
\begin{definition}
A Lie groupoid $\G$ over $\X$ is \defn{source-simply connected} if the fibres of the source map $s : \G \to \X$ are connected and simply connected.  
\end{definition}

A fundamental result in Lie groupoid theory~\cite[Proposition 6.6]{Moerdijk2003} states that if a Lie algebroid has an integration, then it must also have a source-simply connected integration $\G$.  This result comes with an important caveat, however: $\G$ may not be Hausdorff.  
Therefore, as a result of~\autoref{cor:tdint}, we conclude that $\tshf{\X}(-\D)$ has a source-simply connected integration; in~\autoref{sec:unif}, we will demonstrate that in all but one single case, this groupoid is, in fact, Hausdorff.

A key property of source-simply connected groupoids $\G$, shown in~\cite[Proposition 6.8]{Moerdijk2003}, is that they are initial objects: If $\phi : \LA[\G] \to \LA[\H]$ is a morphism to the Lie algebroid of another groupoid $\H$, then there is a unique morphism $\Phi : \G \to \H$ inducing $\phi$.
For this reason, the source-simply connected integration of a Lie algebroid $\sA$, if it exists, is unique up to isomorphism.  We may therefore speak of \emph{the} source-simply connected integration, which we denote by $\FG{\sA}$.
Another consequence of this property is the fact that any representation $\nabla$ of $\sA$ on a vector bundle $\sE$ integrates to a groupoid representation: $\nabla$ may be viewed as a morphism $\sA \to \ati{\sE}$, and so it integrates to a unique morphism $\FG{\sA} \to \GL{\sE}$.  This analogue of Lie's second theorem serves as the crucial integration step in the Riemann--Hilbert correspondence.  We provide a more detailed proof of it below.
\begin{theorem}\label{integrep}
If $\sA$ is an integrable Lie algebroid, then there is a natural equivalence of categories
$$
\Rep{\sA} \cong \Rep{\FG{\sA}}
$$
between representations of $\sA$ and its source-simply connected integration $\FG{\sA}$.
\end{theorem}
\begin{proof}
Let $\G = \FG{\sA}$ and let $\sF\subset\tshf{\G}$ be the tangent sheaf of the foliation defined by the source fibres.  Since $\sA$ is identified with the right-invariant sections of $\sF$, it follows that a flat $\sA$-connection on $\sV$ is equivalent to a flat partial $\sF$-connection on $t^*\sV$.  Furthermore, any $\sA$-flat section $s\in\sV$ lifts to a section $t^*s$ that is flat along $\sF$.

After these preliminaries, let $\sE$ be an $\sA$-representation. Then by the above, $t^*\sE$ has a flat $\sF$-connection.  On the other hand, $s^*\sE$ is constant along $\sF$, so it also has a $\sF$-connection.  Therefore, $\Hom{s^*\sE,t^*\sE}$ has a flat $\sF$-connection. Along the identity submanifold $\X\subset\G$, this bundle coincides with $\End{\sE}$, which has a canonical section $1\in\End{\sE}$.  Since $\X$ intersects each source fibre in a single point, and since the source fibres are simply connected, this canonical section has a unique extension to an $\sF$-flat section 
$$
\Psi \in \cohlgy[0]{\G,\sHom{s^*\sE,t^*\sE}},
$$
which is the required groupoid representation; this procedure defines the integration functor $\Rep{\sA}\to\Rep{\FG{\sA}}$ on objects.

A morphism of $\sA$-representations $\phi:(\sE_1,\nabla_1)\to (\sE_2,\nabla_2)$ is a section of $\sE_1^\vee\otimes\sE_2$ that is flat for the $\sA$-connection $\nabla_1^*\otimes 1+1\otimes\nabla_2$. 
It follows from the preliminaries above that
$t^*\phi$ is $\sF$-flat.  Similarly, $s^*\phi$ is constant along source fibres, so it is trivially $\sF$-flat. Hence, we have a diagram of morphisms that are flat along $\sF$: 
\[
\xymatrix{
s^*\sE_1 \ar[r]^{\Psi_1} \ar[d]_{s^*\phi} & t^*\sE_1 \ar[d]^{t^*\phi} \\
s^*\sE_2 \ar[r]^{\Psi_1} & t^*\sE_2
}
\]
This diagram commutes along the identity submanifold, and since all the maps are flat along $\sF$, we obtain the commutativity on all of $\G$, proving that $\phi$ defines a morphism of groupoid representations.  We omit the details concerning the inverse functor, which is given by differentiating a morphism $\FG{\sA}\to\GL{\sE}$ along the identity submanifold.
\end{proof}
\begin{remark}
We emphasize that the construction of the groupoid representation requires no analysis beyond the classical theorem on existence and uniqueness to solutions of first-order ordinary differential equations.\qed
\end{remark}

\begin{example}
The tangent sheaf $\tshf{\X}$ has integration $\Pair{\X}$, defined in~\autoref{ex:pair}.  The source fibre over a point $x \in \X$ is given by $\X\times \{x\}$ and is generally not simply connected.  The canonical source-simply connected integration is given by the fundamental groupoid $\FG{\X}$ of $\X$.

A representation of $\tshf{\X}$ is simply a vector bundle $\sE$ equipped with a flat connection $\nabla$. The integration of representations referred to above is then simply given by the parallel transport of $\nabla$ along a curve $\gamma : [0,1] \to \X$, which depends only on the homotopy class of $\gamma$ and therefore gives a representation of the fundamental groupoid
$$
\Par{\nabla} : \FG{\X} \to \GL{\sE}.
$$
The parallel transport descends to a representation of $\Pair{\X}$ if and only if the monodromy of the connection is trivial.
\qed
\end{example}

\begin{definition}
Let $(\X,\D)$ be a curve equipped with an effective divisor.  The \defn{fundamental groupoid of $(\X,\D)$} is the source-simply connected integration
$$
\FG{\X,\D} = \FG{\tshf{\X}(-\D)}
$$
of the corresponding Lie algebroid.
\end{definition}
This groupoid will be explicitly described in~\autoref{sec:unif}; we wish to describe here only a few of its basic properties.  In particular, we describe its restriction to $\D$ as well as to the complement of $\D$. Recall that the restriction of a groupoid $\G$ over $\X$ to the subset $\S\subset \X$ is given by the subset $s^{-1}(\S)\cap t^{-1}(\S)$ of arrows with source and target in $\S$.

\begin{lemma}\label{isovsp}
Let $(\X,\D)$ be as above, and let $\U = \X \setminus \D$.  Then
$$
\FG{\X,\D}|_\U \cong \FG{\U}.
$$
If $p \in \D$ is a point of multiplicity $k > 0$, then the isotropy group $\FG{\X,\D}|_p$ at $p$ is the entire source fibre, and is canonically isomorphic to the additive group underlying the vector space $(\ctb[p]{\X})^{k-1}$.
\end{lemma}

\begin{proof}
Since $\U$ is a disjoint union of orbits of the Lie algebroid, it is also a disjoint union of groupoid orbits.  The source fibre of $\FG{\X,\D}|_\U$ at $p\in \U$ is therefore the same as the source fibre of $\FG{\X,\D}$ at $p$, which is simply connected.  Therefore $\FG{\X,\D}|_\U$ is canonically isomorphic to the source-simply connected integration of $\tshf{\X}(-\D) |_\U \cong \tshf{\U}$, which is $\FG{\U}$.

If $p \in \D$ is a point of multiplicity $k$, then $p$ is an orbit of the source-simply connected groupoid $\FG{\X,\D}$.  Therefore, the isotropy group $\FG{\X,\D}|_p$ must coincide with simply-connected Lie group integrating the isotropy Lie algebra of the corresponding Lie algebroid, $\tshf{\X}(-\D)$.  The latter is the one-dimensional abelian Lie algebra $(\ctb[p]{\X})^{k-1}$, which is its own simply-connected integration, giving the result.
\end{proof}

\begin{example}[\airylabel]\label{ex:airy2}
In \autoref{ex:airy-2nd-order}, we viewed the Airy equation as a second-order connection for $\sA = \tshf{\PP[1]}(-3\cdot\infty)$.  Restricting to $\AF[1] = \PP[1] \setminus\{\infty\}$, this operator has kernel spanned by the classical Airy functions $\Ai, \Bi$.  
The corresponding flat connection on $\sE = \sJet[1]{\sO{\AF[1]}}\cong\sO{\AF[1]}\oplus\sO{\AF[1]}$, described in \autoref{ex:airy}, has kernel spanned by
$((\Ai, \Ai'),(\Bi,\Bi'))$, a flat basis of sections for $\sE$ over $\AF[1]$ satisfying the Wronskian condition $\Ai\Bi'-\Bi\Ai' = 1/\pi$.

The source-simply connected integration of $\sA|_{\AF[1]}=\tshf{\AF[1]}$ is $\FG{\AF[1]} \cong \Pair{\AF[1]}$.  Using coordinates $(x,y)$ on $\Pair{\AF[1]}$, the corresponding representation
$$
\Psi_\Airy : s^*\sE  \to t^*\sE
$$
defined by parallel transport can be written explicitly as
\begin{equation}\label{airyunivsol}
\begin{aligned}
\Psi_\Airy(x,y) &= 
\begin{psmallmatrix}
\Ai\phantom{'}(x) & \Bi\phantom{'}(x) \\
\Ai'(x) & \Bi'(x)
\end{psmallmatrix}
\begin{psmallmatrix}
\Ai\phantom{'}(y) & \Bi\phantom{'}(y) \\
\Ai'(y) & \Bi'(y)
\end{psmallmatrix}^{-1} \\
&= \pi \cdot \left(
\begin{smallmatrix}[10pt]
\Ai\phantom{'}(x)\Bi'(y) - \Bi\phantom{'}(x)\Ai'(y) & - \Ai\phantom{'}(x)\Bi(y) + \Bi\phantom{'}(x) \Ai(y) \\
 \Ai'(x)\Bi'(y) - \Bi'(x)\Ai'(y)& -\Ai'(x)\Bi(y)+\Bi'(x)\Ai(y)
\end{smallmatrix}\right).
\end{aligned}
\end{equation}
While the Airy functions, and hence the fundamental solutions, are singular at infinity, the above representation $\Psi_\Airy(x,y)$ extends holomorphically to the groupoid $\FG{\PP[1],3\cdot\infty}$.  We will investigate this extension in~\autoref{airylast}. \qed
\end{example}

\subsection{Birational construction of Lie groupoids}\label{biratgpd}

We now construct a groupoid that integrates the Lie algebroid $\tshf{\X}(-\D)$ on a curve $\X$ with effective divisor $\D$.  We use a technique developed in \cite{Gualtieri2012} for blowing up groupoids along subgroupoids, as follows.       

\begin{theorem}[{\cite[Theorem 2.8]{Gualtieri2012}}]\label{blwpgpd}
	Let $\G$ be a Lie groupoid over the complex manifold $\X$, and let $\H\subset\G$ be a closed Lie subgroupoid over the smooth closed hypersurface $\Y\subset\X$.  Denote by $\Bl{\H}{\G}$ the blowup of $\G$ along $\H$, and let $\S,\T\subset\Bl{\H}{\G}$ be the proper transforms of $s^{-1}(\Y), t^{-1}(\Y)\subset\G$.  Then the manifold
	\begin{equation*}
		\EM{\G}{\H} = \Bl{\H}{\G}\setminus(\S\cup\T)
	\end{equation*}
	is naturally a Lie groupoid over $\X$ such that the blowdown map is a morphism of groupoids $\EM{\G}{\H}\to \G$.

	Furthermore, the Lie algebroid of $\EM{\G}{\H}$ is canonically identified with the subsheaf of $\LA[\G]$ consisting of sections whose restriction to $\Y$ lies in $\LA[\H]$, otherwise known as the elementary modification of $\LA[\G]$ along $\LA[\H]$.   
\end{theorem}
The blowup groupoid satisfies the following universal property, generalizing \cite[Theorem 2.16]{Gualtieri2012} and having the same proof, which we omit.

\begin{proposition}\label{prop:em-univ}
Let $\F$ be a Lie groupoid over $\Z$.  If $\phi : \F\to\G$ is a morphism of groupoids such that $\phi^{-1}(\H) \subset \F$ and $\phi^{-1}(\Y) \subset \Z$ are smooth hypersurfaces, and $\LA[\phi]: \LA[\F] \to \LA[\G]$ factors through $\LA[\EM{\G}{\H}]$, then there is a unique morphism $\phi' : \F \to \EM{\G}{\H}$ making the following diagram commute.
$$
\xymatrix
{ & \EM{\G}{\H} \ar[d] \\
\F \ar[r]_\phi \ar@{-->}[ru]^{\phi'} & \G}
$$
\end{proposition}

The blowup procedure of~\autoref{blwpgpd} implies that any Lie groupoid $\G$ over $\X$ may be modified along a smooth closed hypersurface $\Y \subset \X$: we define the \defn{twist of $\G$ along $\Y$} to be the Lie groupoid
$$
	\G(-\Y) = \EM{\G}{\gid(\Y)}
$$
obtained by blowing up $\G$ along the subgroupoid consisting of the identity elements over $\Y$.  The Lie algebroid of the identity groupoid $\gid(\Y)$ is the zero Lie algebroid, implying that the Lie algebroid of $\G(-\Y)$ is the subsheaf of sections of $\LA[\G]$ that vanish along $\Y$, i.e., $\LA[\G(-\Y)] = \LA[\G](-\Y)$.

More generally, if $\D = k_1 \Y_1 + \ldots + k_n \Y_n$  is an effective divisor supported on the smooth closed hypersurfaces $\Y_1,\ldots,\Y_n$,
we denote by $\G(-\D)$ the groupoid obtained by performing $k_j$ iterated twists of $\G$ along $\Y_j$ for all $1 \le j \le n$.  As above, the Lie algebroid of $\G(-\D)$ is identified with $\LA[\G]$ twisted by $\D$:
$$
	\LA[\G(-\D)] = \LA[\G](-\D).
$$

\begin{definition}
	Let $\X$ be a curve with effective divisor $\D$.  The \defn{pair groupoid of $(\X,\D)$} is the groupoid $\Pair{\X,\D} = \Pair{\X}(-\D)$ obtained by twisting the pair groupoid of $\X$ along $\D$.  
\end{definition}

An immediate consequence of~\autoref{prop:em-univ} is that $\Pair{\X,\D}$ is smallest integration of $\tshf{\X}(-\D)$:
\begin{corollary}
	The groupoid $\Pair{\X,\D}$ is the terminal object in the category of integrations of $\tshf{\X}(-\D)$.
\end{corollary}
Another notable property of $\Pair{\X,\D}$ is that it is an \emph{algebraic} Lie groupoid whenever $\X$ is an algebraic curve.  In contrast, the source-simply connected integration of $\tshf{\X}(-\D)$ is generally non-algebraic.

\begin{remark}
	Suppose that $f : (\X,\C) \to (\Y,\D)$ is a morphism in the sense of~\autoref{cor:mordiv}, so that $\C\geq f^*\D\setminus\R$.  Then by composing $f\times f$ with $(t,s) : \Pair{\X,\C} \to \Pair{\X}$, we obtain a groupoid morphism $\Pair{\X,\C} \to \Pair{\Y}$.
	\autoref{prop:em-univ} then yields a morphism of Lie groupoids
	$$
		f_* : \Pair{\X,\C} \to \Pair{\Y,\D}.
	$$
	On the other hand, any groupoid morphism $F:\Pair{\X,\C} \to \Pair{\Y,\D}$ restricts to the identity bisection to give a morphism $f : (\X,\C)\to(\Y,\D)$ such that $F = f_*$.  Therefore, the assignment $(\X,\D) \mapsto \Pair{\X,\D}$ is a fully faithful functor from the category of effective divisors on curves to the category of Lie groupoids.
\end{remark}

\subsection{Examples}
\label{sec:p1ex}

\subsubsection{Twisted pair groupoids on the affine line}

The pair groupoid $\Pair{\AF[1],k\cdot 0}$ of the affine line twisted by the origin with multiplicity $k>0$ may be described as a subspace of the $k$-fold blowup of $\Pair{\AF[1]}= \AF[1]\times\AF[1]$ in the following way. 

Begin with coordinates $(z,u)$ on $\Pair{\AF[1]}$ so that the groupoid structure is:  
$$
\begin{aligned}
s(z,u)&= z\\
t(z,u)&= z+u\\
(z_2,u_2)\cdot (z_1,u_1) &= (z_1,u_2+u_1).
\end{aligned}
$$
According to~\autoref{blwpgpd}, to construct $\Pair{\AF[1],0}$, we must blow up at $\gid(0) = (0,0)$ and remove the proper transforms $\S, \T$ of $z=0$ and $z+u=0$, respectively.  The affine chart $(z',u')$ for the blowup such that $z=z'$ and $u=z'u'$ is convenient as it excludes $\S$.  In this chart, $\T =\{1+u'=0\}$; hence $\Pair{\AF[1],0}$ is identified with the complement of $1+u'=0$ in $\AF[1]\times\AF[1]$, with groupoid structure
$$
\begin{aligned}
s(z',u')&=z'\\
t(z',u')&=(1+  u')z'\\
(z'_2,u'_2)\cdot(z'_1,u'_1) &= (z_1, u_2'(1 + u_1')+u_1').
\end{aligned}
$$
We see immediately that the map $(z',u')\mapsto (z', 1+u')$ defines an isomorphism from $\Pair{\AF[1],0}$ to the standard action groupoid $\CC^*\ltimes\AF[1]$ over $\AF[1]$.

To construct $\Pair{\AF[1],2\cdot 0}$, we blow up $\gid(0)=(0,0)$ in the already constructed $\Pair{\AF[1],0}$ and remove the proper transform of $z'=0$.  The result is again covered by a single affine chart $(z'',u'')$ such that $z' = z''$ and $u'=z''u''$.  Continuing in this way, we obtain all twists of the pair groupoid of $\AF[1]$ at zero:
\begin{theorem}\label{paircords}
The twisted pair groupoid $\Pair{\AF[1],k\cdot 0}$ may be described as the complement of the curve $1 + uz^{k-1} = 0$ in $\AF[1]\times \AF[1]$, with groupoid structure
$$
\begin{aligned}
s(z,u)&=z\\
t(z,u)&=(1 + uz^{k-1})z \\
(z_2,u_2)\cdot(z_1,u_1) &= (z_1, u_2(1 + u_1z_1^{k-1})^k + u_1).
\end{aligned}
$$
\end{theorem}

\begin{remark}
	Although $\Pair{\AF[1],k\cdot 0}$ is an action groupoid for $k=1$, it is not so for $k>1$, because the vector field $z^k\cvf{z}$ fails to be complete for $k>1$: integration from initial position $z$ reaches a singularity in finite time, corresponding to the proper transforms that must be removed to form the groupoid.
\qed
\end{remark}

\subsubsection{Stokes groupoids}

The same technique used to construct the twisted pair groupoids may be used to construct the source-simply connected groupoids $\FG{\AF[1],k\cdot 0}$ integrating $\tshf{\AF[1]}(-k\cdot 0)$, which we call the \defn{Stokes groupoids} and denote by $\Sto{k}$.

First, we observe that $\Sto{1}=\FG{\AF[1],0}$ is the action groupoid $\CC\ltimes\AF[1]$ associated to the exponential action, which in coordinates $(z,u)$ means that 
$$
\begin{aligned}
s(z,u)&=z\\
t(z,u)&=\exp(u) z \\
(z_2,u_2)\cdot(z_1,u_1) &= (z_1,u_2+u_1).
\end{aligned}
$$
Performing the iterated blowup in $\gid(0) = (0,0)$ as before, we obtain
\begin{theorem}\label{thm:twfunline}
The Stokes groupoid $\Sto{k}$ may be identified with $\CC\times \AF[1]$, with groupoid structure
$$
\begin{aligned}
s(z,u)&=z\\
t(z,u)&=\exp({u z^{k-1}})z \\
(z_2,u_2)\cdot(z_1,u_1) &= (z_1,u_2\exp({(k-1)u_1z_1^{k-1}}) + u_1).
\end{aligned}
$$
\end{theorem}
Since $\Sto{k}$ is source-simply connected, there is a canonical groupoid homomorphism $E:\Sto{k}\to \Pair{\AF[1],k\cdot 0}$, given by  
\begin{equation}\label{grpexp}
	E(z,u) =
	\left(z, \frac{e^{u z^{k-1}}-1}{z^{k-1}}\right). 
\end{equation}
Furthermore, the natural groupoid homomorphism $\Sto{k+1}\to\Sto{k}$ integrating the inclusion of $\tshf{\AF[1]}(-(k+1)\cdot 0)$ in $\tshf{\AF[1]}(-k\cdot 0)$ is simply given by $(z,u)\mapsto (z,u z)$.

\subsubsection{Degree one twisted pair groupoid on the projective line}\label{sec:p1d1}
We now focus on the effective divisor $0\in \PP[1]$ of degree one.  To construct $\Pair{\PP[1],0}$, we blow up the pair groupoid $\Pair{{\PP[1]}} = \PP[1]\times \PP[1]$ along $\gid(0)=(0,0)$, and remove the rational curves $\S',\T'$ given by the proper transforms of $\S = \PP[1]\times\{0\}$ and $\T = \{0\}\times \PP[1]$, as illustrated in~\autoref{fig:p1blowup} 
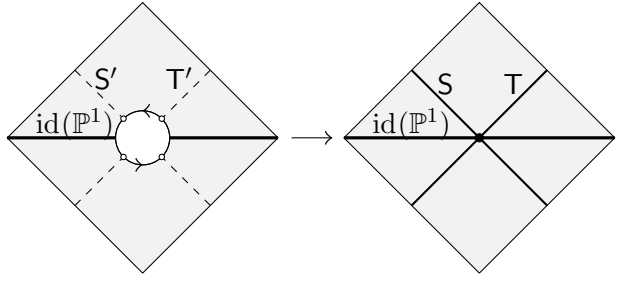
\begin{figure}
\centering
\begin{tikzpicture}[scale=1.8]
\filldraw[fill=black!5!white] (0,-1)--(1,0)--(0,1)--(-1,0)--(0,-1);
\draw[very thick] (-1,0) -- (-.2,0);
\draw[very thick] (.2,0)-- (1,0);
\draw[dashed] (-0.5,0.5) -- (-0.1414,0.1414);
\draw[dashed] (0.5,0.5) -- (0.1414,0.1414);
\draw[dashed] (0.5,-0.5) -- (0.1414,-0.1414);
\draw[dashed] (-0.5,-0.5) -- (-0.1414,-0.1414);
\draw[fill=white] (0,0) circle (.2);

\draw (0.27,0.44) node {$\T'$};
\draw (-0.27,0.44) node {$\S'$};
\draw (-0.5,0.1) node {$\gid(\PP[1])$};
\draw[decoration={markings,mark=at position 1 with {\arrow[scale=1.5]{>}}},
    postaction={decorate}] (1.1,0) -- (1.4,0);
\draw[decoration={markings,mark=at position 1 with {\arrow[scale=1.5]{>}}},
    postaction={decorate}] (0,.2) arc (90:270:.2);
\draw[decoration={markings,mark=at position 1 with {\arrow[scale=1.5]{>}}},
    postaction={decorate}]   (0,-.2) arc (270:450:.2);
\draw[fill=white] (0.1414,0.1414) circle ({0.02});
\draw[fill=white] (0.1414,-0.1414) circle ({0.02});
\draw[fill=white] (-0.1414,0.1414) circle ({0.02});
\draw[fill=white] (-0.1414,-0.1414) circle ({0.02});    
\end{tikzpicture}
\begin{tikzpicture}[scale=1.8]
\filldraw[fill=black!5!white] (0,-1)--(1,0)--(0,1)--(-1,0)--(0,-1);
\draw[very thick] (-1,0) -- (1,0);
\draw[thick] (-0.5,0.5) -- (0.5,-0.5);
\draw[thick] (-0.5,-0.5) -- (0.5,0.5);
\draw (0.25,0.4) node {$\T$};
\draw (-0.25,0.4) node {$\S$};
\draw (-0.5,0.1) node {$\gid(\PP[1])$};
\draw[fill=black]  (0,0) circle (0.03);
\end{tikzpicture}
\parbox{.82\textwidth}{\caption{The groupoid $\Pair{{\PP[1]},0}$, presented as the blowup of $\PP[1]\times\PP[1]$ at $(0,0)$ with the proper transforms of $\S=\PP[1]\times\{0\}$ and $\T=\{0\}\times\PP[1]$ deleted.}\label{fig:p1blowup}}
\end{figure}

For a clearer geometric description of the groupoid structure, blow down the disjoint $-1$-curves $\S'$ and $\T'$ to obtain a copy of $\PP[2]$ with two privileged points $p_s,p_t \in \PP[2]$.  This gives an isomorphism
$$
\Pair{{\PP[1],0}} \cong \PP[2] \setminus \{p_s,p_t\}.
$$
The identity bisection is then a line in $\PP[2]$ disjoint from $p_s$ and $p_t$, and the source and target fibres of the groupoid coincide with pencils of lines through $p_s$ and $p_t$ respectively.  These fibres are transverse, except when they coincide with the line $\overline{p_s p_t}$, which intersects the identity bisection in the original divisor $0$.  The groupoid composition is given by the geometric construction illustrated in \autoref{fig:p2gpd}.
\begin{figure}[h]
\centering
\begin{tikzpicture}[scale=1.8]
\def\p{0.8}
\def\r{0.02}
\clip[draw] (0,0) circle (1);
\filldraw[fill=black!5!white] (1,0) arc (0:360:1);
\draw[very thick] (-1,0) -- (1,0);
\draw[dotted] (0,-1) -- (0,1);
\coordinate [label=right:$h$] (g) at (-0.3,0.2);
\coordinate [label=left:$g$] (h) at (-0.5,0.2);
\coordinate [label=right:$p_s$](ps) at (0,{-\p});
\coordinate [label=right:$p_t$](pt) at (0,{\p});
\draw[fill=black] (h) circle (\r);
\draw[fill=black] (g) circle (\r);
\draw[name=sg, dotted] (ps) -- ($ (ps) ! 3.5 ! (g) $);
\draw[name=sh, dotted] (ps) -- ($ (ps) ! 3.5 ! (h) $);
\draw[name=tg, dotted] (pt) -- ($ (pt) ! 3.5 ! (g) $);
\draw[name=th, dotted] (pt) -- ($ (pt) ! 3.5 ! (h) $);
\coordinate [label=above:$gh$](gh) at (-0.35,0.375);
\draw[fill=black]  (gh) circle (\r);
\draw (-0.6,-0.15) node {$\gid(\PP[1])$};
\draw[fill=black]  (0,0) circle (0.03);
\draw (0.25,-0.15) node {$\gid(0)$};
\draw[fill=white] (0,{\p}) circle ({\r});
\draw[fill=white] (0,{-\p}) circle ({\r});
\end{tikzpicture}
\parbox{.82\textwidth}{\caption{The groupoid $\Pair{{\PP[1],0}}$, presented as $\PP[2]$ with a distinguished line $\gid(\PP[1])$ and punctured in two points $p_s, p_t$ not meeting the line.  
}\label{fig:p2gpd}}
\end{figure}
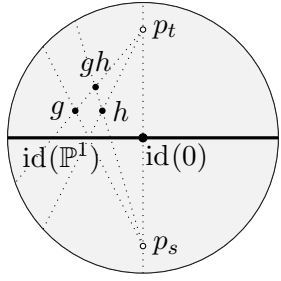
Finally, note that while the generic source fibre for $\Pair{\PP[1],0}$ is isomorphic to the affine line, the fibre over $0\in\PP[1]$ is a $\CC^*$ isotropy group, and hence this groupoid is not source-simply connected. 

\subsubsection{Reduced degree two divisor on the projective line}
\label{sec:p1d2}
Consider the reduced divisor $\D = 0+\infty$ on $\PP[1]$.  To construct $\Pair{\PP[1],\D}$, we blow up $\Pair{\PP[1]}$ in $\gid(0)$ and $\gid(\infty)$ and delete the proper transforms of the source and target fibres $\S_0,\S_\infty,\T_0,\T_\infty$ over $0,\infty$.

These proper transforms $\S_0',\S_\infty',\T_0',\T_\infty'$ are $-1$-curves; by blowing down a disjoint pair of them, say $\S_0', \S_\infty'$, we recover $\PP[1]\times\PP[1]$ and identify  
$$
\Pair{{\PP[1],0+\infty}} \cong \PP[1] \times \PP[1] \setminus \rbrac{\T_0'' \cup \T_\infty''}
$$
where $\T_0''$ and $\T_\infty''$ are the images of $\T_0'$ and $\T_\infty'$ under the blow-down.  These curves may be identified with the fibres over $0$ and $\infty$ of the first projection $\PP[1]\times\PP[1]\to\PP[1]$.  In this way, we identify
$$
\Pair{{\PP[1],0+\infty}} \cong \CC^* \ltimes \PP[1]
$$
where $\CC^* \ltimes \PP[1]$ is the groupoid associated to the action of $\CC^*$ on $\PP[1]$ fixing $0$ and $\infty$.  This groupoid is not source-simply connected, as its the source fibres are all isomorphic to $\CC^*$.  However, composing with the exponential map, we obtain an action of $\CC$ on $\PP[1]$, and the source-simply connected integration is
$$
\FG{\PP[1],0+\infty} \cong \CC \ltimes \PP[1].
$$

\subsubsection{Confluent degree two divisor on the projective line}
\label{sec:p1-2}
To obtain the groupoid associated to the divisor $2\cdot 0$ on $\PP[1]$, we must twist the pair groupoid twice along $\gid(0)$.  The first twist, performed in \autoref{sec:p1d1}, yielded
\[
\Pair{{\PP[1],0}}\cong \PP[2] \setminus \{p_s,p_t\},
\] 
and the second twist is obtained by blowing the first one up at $\gid(0)$, which is the intersection of the lines $\overline{p_sp_t}$ and $\gid(\PP[1])$.  Blowing up $\PP[2]$ at $\gid(0)$, we obtain a Hirzebruch surface $\Sig_1 \cong \PP(\sO{} \oplus \sO{}(-1))$, which fibres over the exceptional divisor via $\pi:\Sig_1\to\E$.  The proper transforms of the lines $\gid(\PP[1])$ and $\overline{p_sp_t}$ are then the fibres over the special points $p_0, p_\infty\in \E$, respectively. The groupoid is therefore given by the complement of $\pi^{-1}(p_\infty)$,
$$
\Pair{{\PP[1],2\cdot 0}} \cong \Sig_1 \setminus \pi^{-1}(p_\infty),
$$
with identity bisection given by $\pi^{-1}(p_0)$. The source and target fibres are then given by the two pencils of sections of $\pi$ passing through the images of the points $p_s, p_t$ in $\pi^{-1}(p_\infty)$. Note that all source fibres are isomorphic to the affine line, and so this groupoid is actually source-simply connected.


\subsubsection{Confluent divisors of arbitrary degree on the projective line}

The groupoid $\Pair{\PP[1],2\cdot 0}$ has an alternative description as follows: the Lie algebroid $\sA = \tshf{\PP[1]}(-2\cdot 0)$ has a one-dimensional space of global sections, given by the infinitesimal generators of the group of parabolic automorphisms of $\PP[1]$ that fix $0$.  Therefore the exponential map
$$
\cohlgy[0]{\PP[1],\sA} \to \Aut{{\PP[1]}}
$$
is an algebraic embedding, giving rise to an action of the additive group of sections of $\sA$ on $\PP[1]$ and so an identification with the action groupoid:
$$
\Pair{\PP[1],2\cdot 0} \cong \cohlgy[0]{\PP[1],\sA} \ltimes \PP[1].
$$
The evaluation map
$\cohlgy[0]{\PP[1],\sA} \otimes \sO{\PP[1]} \to \sA$
is an isomorphism, and hence the groupoid is identified with the total space of $\tshf{\PP[1]}(-2\cdot 0)$ in such a way that the source map is the bundle projection and the identity map is the zero section.  
We now perform an iterated twist of this groupoid to obtain $\Pair{\PP[1],k\cdot 0}$ with $k >2$.
\begin{theorem}\label{thm:p1-k}
For $k \ge 2$ and $p \in \PP[1]$, there is a canonical isomorphism
$$
\Pair{\PP[1],k\cdot p} \cong \Tot{\tshf{\PP[1]}(-k\cdot p)},
$$
for which the source map is the bundle projection and the identity bisection is the zero section.  In particular, $\Pair{\PP[1],k\cdot p}$ is source-simply connected, so it coincides with $\FG{\PP[1],k\cdot p}$.  Moreover, the groupoid homomorphisms
$\Pair{\PP[1],k\cdot p} \to \Pair{\PP[1],k'\cdot p}$
for $k \ge k' \ge 2$ are given by the natural maps
$$
\Tot{\tshf{\PP[1]}(k\cdot p)} \to \Tot{\tshf{\PP[1]}(k'\cdot p)}.
$$
\end{theorem}

\begin{proof}
Assuming that the statement is true for $k \ge 2$, the source and target fibres of $\Pair{\X,k\cdot p}$ over $p$ are equal to the fibre of $\tshf{\PP[1]}(-k\cdot p)$ over $p$.  Therefore, $\Pair{\X,(k+1)\cdot p}$ is obtained by blowing up $0_p \in \Tot{\tshf{\PP[1]}(-k\cdot p)}$ and removing the proper transform of the fibre.  By the elementary \autoref{lem:blow-line} below, the result is $\Tot{\tshf{\PP[1]}(-(k+1)\cdot p)}$.  The statements about the source and identity maps are true by continuity, since the blow-down is an isomorphism away from the fibre over~$p$.
\end{proof}

\begin{lemma}\label{lem:blow-line}
Let $\X$ be a curve, $\sL$ a line bundle on $\X$ and $p \in \X$.  Let $\Y = \Bl{\Tot{\sL}}{0_p}$ be the blowup of the total space of $\sL$ at the zero element over $p$, and let ${\sL|'_p} \subset \Y$ be the proper transform of the fibre over $p$.  Then there is a canonical identification
$$
\Y \setminus {\sL|'_p} \cong \Tot{\sL(-p)},
$$
for which the blow down is the canonical map $\Tot{\sL(-p)}\to\Tot{\sL}$.
\end{lemma}


%% file: stokes1-unif.tex

\subsection{Gluing over orbit covers}

Heuristically, the source-simply connected groupoid $\FG{\X,\D}$ integrating the Lie algebroid $\tshf{\X}(-\D)$ consists of all homotopy classes of paths in the complement $\U=\X\setminus\D$, together with a space of limiting paths over each point in $\D$. Such groupoids are nonlocal objects over $\X$. However, it is possible to construct them by gluing local pieces defined over an open covering of $\X$, if this covering is chosen in a manner compatible with the groupoid structure.

An \defn{orbit cover} of $\X$ is an open covering such that each orbit of the Lie algebroid is completely contained in at least one of the open sets in the cover.  Since the orbits of $\tshf{\X}(-\D)$ are the points of $\D$ and the complement of $\D$, a convenient choice of orbit cover is $\{\U,\V\}$, where $\U = \X\setminus\D$ and $\V$ is a union of disjoint open disc neighbourhoods of all the points of $\D$.  Restricting the groupoid to such a covering
leads to the following diagram of groupoid inclusions:
\begin{equation}\label{copgpd}
	\xymatrix@ur{ 
	 \FG{\X,\D}\mathrlap{|^c_\U\ar[r]} & \FG{\X,\D}\\
	 \FG{\X,\D}\mathrlap{|^c_{\U\cap\V}}\ar[u]^i\ar[r]_j & \FG{\X,\D}\mathrlap{|^c_\V\ar[u]} 
	}
\end{equation}
where $\FG{\X,\D}|^c_\U$ denotes the source-connected component of the identity bisection in the restriction of the groupoid to $\U$, and similarly for $\V, \U\cap\V$.
As shown in~\cite[Theorem 3.4]{Gualtieri2012}, if the covering is an orbit cover, the above diagram is a pushout diagram in manifolds; in other words, $\FG{\X,\D}$ may be constructed by gluing its restrictions on $\U,\V$, which are obtained as follows.  

The restriction $\FG{\X,\D}|^c_\U$ is the source-simply connected integration of the tangent sheaf $\tshf{\U}$ of the orbit $\U$, and so coincides with the fundamental groupoid $\FG{\U}$.  The restriction $\FG{\X,\D}|^c_{\U\cap\V}$, however, coincides with $\FG{\U\cap\V}$ only when each of the generating loops in the fundamental groups of the annuli comprising $\U\cap \V$ maps nontrivially into $\U$; this condition is satisfied by all divisors on all connected curves, except for the case $\X = \PP[1]$ with $\D$ supported at a single point (meaning $\U$ is contractible).  

Let $\V_p\subset\V$ be the neighbourhood of a point $p\in\D$ of multiplicity $v(p)$, and let $\phi_p:\V_p\to\DD$ be a coordinate chart.  With the same exception as above, $\phi_p$ induces an isomorphism from $\FG{\X,\D}|^c_{\V_p}$ to the restriction to the unit disc $\DD$ of the Stokes groupoid described in~\autoref{thm:twfunline}. That is, 
$$
\phi_p:\FG{\X,\D}|^c_{\V_p} \stackrel{\cong}{\longrightarrow} \Sto{v(p)}|_{\DD}.
$$ 
Restricting the maps $\{\phi_p\}$ to $\U\cap\V$, we obtain an isomorphism over a union of punctured discs:
\begin{equation}\label{eq:glugrp}
	\varphi: \FG{\U\cap\V}\stackrel{\cong}{\longrightarrow} \coprod_{p\in\D}\Sto{v(p)}|_{\DD^*}.
\end{equation}
\begin{theorem}
	Let $\X$ be a connected curve equipped with an effective divisor $\D = \sum_{p\in\X} v(p) p$, and assume $\U=\X\setminus\D$ is not contractible.  The twisted fundamental groupoid of $(\X,\D)$ is then
	$$ 
		\FG{\X,\D} \cong \FG{\U}\cup_\varphi 
			\coprod_{p\in\D} \Sto{v(p)}|_{\DD}, 
	$$
	where $\varphi$ is the identification of the open subgroupoids given by~\autoref{eq:glugrp}.
\end{theorem}
Since coproducts of Hausdorff groupoids over orbit covers are always Hausdorff~\cite[Proposition 3.7]{Gualtieri2012}, we conclude that the groupoids $\FG{\X,\D}$ are Hausdorff whenever $\U$ is not contractible. But we proved in~\autoref{thm:p1-k} that $\FG{\PP[1],k\cdot p}$ is Hausdorff for $k\geq 2$. So, the only remaining case to investigate is the degree one divisor $\D=p$ on the projective line.

The groupoid $\FG{\PP[1],p}$ may be described as a fibred coproduct in the following way.
By choosing a chart $\phi_p:\V\to\DD$ centred on $p$, we obtain an embedding $j:\FG{\U\cap\V}\to \Sto{1}|_{\DD^*}$.  On the other hand, we have a natural map $i:\FG{\U\cap\V}\to\FG{\U}\cong \U\times\U$.  Then 
$$
\FG{\PP[1],p} \cong (\U\times\U){_i\cup}_j \Sto{1}|_{\DD}.
$$
This coproduct fails to be Hausdorff since $i$ is not an embedding.  More concretely, while the various classes of loops surrounding the origin in $\Sto{1}$ limit to distinct points in the isotropy group over $0$, they are all forced to be equivalent in $\U$.

\subsection{Embedding in the uniformizing projective bundle}\label{sec:unifprojbdle}

In this section we give a global description of the groupoid $\FG{\X,\D}$ as an open subset of a projective line bundle over $\X$ that is equipped with a meromorphic projective connection in the sense of~\autoref{sec:meroproj}.  This projective connection derives from a uniformization of the complement $\U=\X\setminus\D$, as follows.  We focus on the most prevalent case, where $\U$ has universal cover the upper half plane $\HH$; the few remaining cases can be dealt with similarly.  

Let $\pi:\HH\to\U$ be a uniformization, and let $\tau$ be the standard coordinate on $\HH$.  The upper half plane is endowed with a canonical projective structure, and hence a projective connection on the square root $\omega^{-1/2}_\HH$ of the anti-canonical sheaf, given by 
\begin{equation}\label{unifeqn}
	\tilde{D}(fd\tau^{-1/2}) = \frac{d^2\!f}{d\tau^2} d\tau^{3/2}.
\end{equation}
This operator is invariant under the action of $\PGL{2,\RR}$, including the Fuchsian subgroup of deck transformations for the covering $\pi$, and hence it descends to a projective connection $D$ on $\omega^{-1/2}_\U$.  We now review and extend an observation of Kra~\cite[Section~1.5]{Kra1989}:

\begin{lemma}\label{uniconnection}
	The uniformizing projective connection on $\U=\X\setminus\D$ extends uniquely to a meromorphic projective connection on $\X$ with poles bounded by the reduced divisor $\D_\red\subset \D$, i.e., a projective connection for $\tshf{\X}(-\D_\red)$.
\end{lemma}
\begin{proof}
For each $p\in\D$, we may choose a coordinate $z$ centered at $p$ so that over the punctured disc $\DD^*=\{z : 0<|z|<1\}$, the covering $\pi:\HH\to\U$ is equivalent to the quotient by a parabolic element, and so to the map $\tau\mapsto z= \exp(2\pi i \tau)$. 

The initial projective connection has coefficient $q_\tau = 0$ in the sense of~\eqref{conndef}, and by~\eqref{projchan} we compute
that $q_z = \tfrac{1}{4} z^{-2}$, so the pushforward connection is
\begin{equation}\label{regform}
D(f dz^{-1/2}) = (f''  + \tfrac{1}{4} z^{-2} f) dz^{3/2},
\end{equation}
which is the form of a projective $\sA$-connection for $\sA = \tshf{\DD}(-0)$ and hence for the reduced divisor, as required.  
\end{proof}

The projective connection for $\tshf{\X}(-\D_{\red})$ described above naturally induces a projective connection for $\tshf{\X}(-\D)$, by transporting the associated first order connection via the morphism of jet bundles $\varphi$:
$$ 
\xymatrix@R=1.5em@C=1.5em{0\ar[r] &\tshf{\X}(-\D_\red)^\vee\otimes\omega_\X^{-1/2}\ar[r]\ar[d] & \sJet[1][\tshf{\X}(-\D_{\red})]\omega^{-1/2}_\X\ar[r]\ar[d]^{\varphi} & \omega_\X^{-1/2}\ar[r]\ar@{=}[d] & 0\\
0\ar[r] &\tshf{\X}(-\D)^\vee\otimes\omega_\X^{-1/2}\ar[r] & \sJet[1][\tshf{\X}(-\D)]\omega^{-1/2}_\X\ar[r] & \omega_\X^{-1/2}\ar[r] & 0
}
$$
The passage from the $\tshf{\X}(-\D_{\red})$ jet bundle to the $\tshf{\X}(-\D)$ jet bundle is a composition of elementary modifications at the points of the divisor $\D$.  To see how this modification affects the projective connection, we make the following local computation. 

If $D(f) = \rbrac{\delta^2 f + p_{1}\delta f + p_0f} \alpha^2$ is a second order connection for $\tshf{\AF[1]}(-n\cdot 0)$, written in terms of the dual bases $\delta = z^n\cvf{z}$ and $\alpha = z^ndz$, then by conjugating its associated first order connection~\eqref{eqn:jet-conn} by the morphism $\varphi$ from $\tshf{\AF[1]}(-n\cdot 0)$-jets to $\tshf{\AF[1]}((-n-1)\cdot 0)$-jets, we obtain 
\begin{equation}\label{modcnx}
\varphi \nabla \varphi^{-1} = d + 
\begin{pmatrix*}
    0 & -1  \\
    z^2p_{\mathrlap{0}} & zp_{\mathrlap{1}}-z^n  
  \end{pmatrix*} z^{-n-1}dz,
\end{equation}
which corresponds to the second order $\tshf{\AF[1]}(-(n+1)\cdot 0)$--connection
$$
D^\varphi(f) = \rbrac{\delta^2 f + (zp_{1}-z^n)\delta f + (z^2p_0)f} \alpha^2,
$$
where now $\delta = z^{n+1}\cvf{z}$ and $\alpha = z^{-n-1}dz$.
\begin{definition}
	The uniformizing projective bundle $\P(\X,\D)$ is defined by
	\begin{equation*}
		\P(\X,\D)= \PP(\sJet[1][\tshf{\X}(-\D)]{\can[\X]^{-1/2}}),
	\end{equation*}
	equipped with the Ehresmann $\tshf{\X}(-\D)$-connection induced by the projective connection of~\autoref{uniconnection} with the elementary modifications described above.
\end{definition}

Since $\P(\X,\D)$ is equipped with a $\tshf{\X}(-\D)$-connection, there is a natural action of the source-simply connected integration $\FG{\X,\D}$ on $\P(\X,\D)$, i.e., a homomorphism of groupoids 
$$
\Psi : \FG{\X,\D} \to \PGL{\P(\X,\D)}.
$$
The bundle $\P(\X,\D)$ has a canonical section $\sigma : \X \to \P(\X,\D)$, given by the jet sequence.  Acting on this section, we obtain the following map.
\begin{align}
\mapdef{\iota}{\FG{\X,\D}}{\P(\X,\D)}
			{ g } { \Psi(g)\cdot \sigma(s(g)) }\label{eqn:embed-def}
\end{align}
In the remainder of this section, we show that the map $\iota$ is an open embedding.
\begin{theorem}
The source-simply connected groupoid $\FG{\X,\D}$ integrating $\tshf{\X}(-\D)$ is naturally identified with the open subset of the uniformizing projective bundle $\P(\X,\D)$ formed by the union of all leaves of the horizontal foliation that intersect the image of the section $\sigma:\X\to\P(\X,\D)$.  

In this embedding, the source and target fibres coincide with the
horizontal and vertical foliations, respectively, and the identity bisection coincides with $\sigma$.
\end{theorem}

\begin{proof}
We first show that the source fibres of $\FG{\X,\D}$ are mapped by $\iota$ isomorphically onto leaves of the horizontal foliation of $\P(\X,\D)$.

The action of the source fibre through $p \in\X$ is given by applying the parallel transport of the $\PGL{2,\CC}$-connection to the point $\sigma(p) \in \P(\X,\D)$.  Hence, the source fibres must map to horizontal leaves.  Suppose that $p \in \U = \X \setminus \D$.  Since $\FG{\X,\D}|_\U = \FG{\U}$ is the fundamental groupoid of $\U$, the source fibre through $p$ is a universal covering of $\U$.  Moreover, $\P(\X,\D)|_\U$ is identified with the uniformizing projective structure for this covering.  It follows that the action by this source fibre consists of lifting homotopy classes of paths in $\U$ to a universal cover of $\U$ and is therefore an isomorphism.

It remains to check the source fibres over $\D$.  Let $z$ be a coordinate centered at a point $p\in\D$ of multiplicity $k$.  Applying the modification \eqref{modcnx} to the connection~\eqref{regform} for the reduced divisor, we see that the connection on the jet bundle near $p$ has the form
\begin{equation}\label{locprojcon}
\nabla = d + \begin{pmatrix}
0 & -1 \\
\frac{1}{4} z^{2(k-1)} & -kz^{k-1}
\end{pmatrix}z^{-k}dz,
\end{equation}
in the basis $(dz^{-1/2}, z^{-k}dz\otimes dz^{-1/2})$.  
By~\autoref{isovsp}, the isotropy group of $\FG{\X,\D}$ over $p$ may be identified with additive group of complex numbers, which acts on the fibre $\P(\X,\D)|_p$ by the exponential of the leading order coefficient in~\eqref{locprojcon} at $z=0$, namely
\begin{equation}\label{residuek} 
A = \begin{pmatrix}0&-1\\0&0\end{pmatrix}\text{ for } k>1,\hspace{3em}
A = \begin{pmatrix}0&-1\\\tfrac{1}{4}&-1\end{pmatrix}\text{ for } k=1.
\end{equation}
In both cases, the trace-free part of $A$ has rank $1$ and satisfies
$$
A \cdot \begin{pmatrix}
0 \\ 1
\end{pmatrix} \ne 0,
$$
implying the isotropy action is by parabolic translations of the projective fibre and that $\sigma(p)$ is not a fixed point. Hence the isotropy group acts freely on the orbit through $\sigma(p)$, as required.

Therefore, $\iota$ maps each source fibre injectively into
a horizontal leaf of $\P(\X,\D)$.  To see that $\iota$ is injective, it therefore suffices to show that distinct source fibres map to distinct horizontal leaves.  But this fact is an immediate consequence of the fact that the sheaf of flat sections of $\can[\X]^{-1/2}$ over $\U$ has a basis $\eta_1,\eta_2$ given by multisections whose ratio $\eta_2/\eta_1 = \tau$ is an inverse to the universal covering map $\HH \to \U$.

We conclude that $\iota$ is an injective holomorphic map between complex manifolds of equal dimension, and so it is an open embedding.  The statements concerning the groupoid structure follow directly from the formula \eqref{eqn:embed-def} defining $\iota$.
\end{proof}

\begin{remark}
	The elementary modifications used to obtain the connection on the $\tshf{\X}(-\D)$ jet bundle may be described purely as operations on projective bundles: $\P(\X,\D)$ is obtained from $\P(\X,\D_\red)$ by a series of blowups and blowdowns as follows. If $p\in\C$, then we obtain $\P(\X,\C+p)$ by first blowing up $\P(\X,\C)$ at $\sigma(p)$, then blowing down the proper transform of the fibre $\P(\X,\C)|_p$.  
	Note that this operation on projective bundles, when restricted to the image of the open embedding $\iota$, coincides with the blowup operation on groupoids introduced in~\autoref{biratgpd}.\qed
\end{remark}

\begin{remark}
An alternative description of the image of $\iota$ in $\P(\X,\D)$ is as follows.  Over each point $p \in \D$, the image intersects the projective fibre in the complement of a single point---the unique fixed point for the action of the isotropy group.  
Over $\U=\X\setminus\D$, the projective bundle $\P(\X,\D)$ inherits a fibrewise Hermitian metric from the uniformization $\pi:\HH\to\U$; the image of $\iota$ may then be described as the set of points within a fixed distance of the section $\sigma$.\qed
\end{remark}

\subsection{The hypergeometric groupoid}

The hypergeometric family of equations, parametrized by $a,b,c\in\CC$, is given by the second order ODE
\begin{equation}\label{gausshpg}
z(1-z) f'' + (c - (a+b+1)z) f' - ab f = 0.
\end{equation}
We may write this equation as $Df = 0$, where $D$ is a second order $\sA$-connection on $\sO{\PP[1]}$ and the algebroid $\sA = \tshf{\PP[1]}(-\D)$ is associated to the divisor $\D=0+1+\infty$ over which the equation has regular singular points.  In a basis $\delta = z(1-z)\cvf{z}$ for $\sA$ over $\AF[1]=\PP[1]\setminus\{\infty\}$ (and the dual basis $\alpha$ for $\sA^\vee$), we have 
\begin{equation} \label{hyprgm}
Df = (\delta^2 f + p_1 \delta f + p_0 f)\alpha^2, 
\end{equation}
where $p_1 = c-1 - (a+b-1)z$ and $p_0 = -abz(1-z)$. 
The corresponding $\sA$-representation~\eqref{eqn:jet-conn} is then an $\sA$-connection on $\sJet[1][\sA]\sO{\PP[1]}\cong \sO{\PP[1]}\oplus\sO{\PP[1]}(1)$.

We now give a global description of the \defn{hypergeometric groupoid} $\FG{\PP[1],\D}$ which serves, in particular, as the universal domain of definition for the fundamental solutions to all equations in the hypergeometric family described above. As shown in~\autoref{sec:unifprojbdle}, it is an open subset of the projective bundle $\PP(\sJet[1][\sA]{\can[{\PP[1]}]^{-1/2}}) \cong \PP(\sO{}(1)\oplus\sO{}(2))$, with target map given by the bundle projection and identity bisection given by the inclusion map in the jet sequence.  The source map is described as follows.  

Let $\U = \PP[1]\setminus\D$, and recall that the universal covering map $\HH \to \U$ is given by the elliptic modular function $\lambda(\tau)$.  Using \eqref{conndef} and \eqref{projchan}, the induced projective connection on $\U$ is given by
$$
D( f dz^{-1/2} ) = \rbrac{\deriv{^2\,f}{z^2} + \frac{z^2-z+1}{4z^2(1-z)^2}f} dz^{3/2}.
$$
This operator has an identical form in the coordinate at infinity, and so extends to a projective $\sA$-connection. In fact, it coincides with a particular hypergeometric operator~\eqref{hyprgm} with $a = b = \tfrac{1}{2}$ and $c=1$, twisted by the rank one $\sA$-representation $\can[{\PP[1]}]^{-1/2}\cong\sO{\PP[1]}(1)$ with flat multisection $\sqrt{z(1-z)}dz^{-1/2}$.  The standard solution $d\tau^{-1/2}$ on $\HH$ pushes down to the multisection
$$   
\eta_1=(\pi i)^{1/2} F(\tfrac{1}{2},\tfrac{1}{2},1; z)\sqrt{z(1-z)}dz^{-1/2},
$$
and so the other solution $\tau d\tau^{-1/2} = (d(-\tau^{-1}))^{-1/2}$ pushes down to  
$$
\eta_2=i (\pi i)^{1/2}F(\tfrac{1}{2},\tfrac{1}{2},1; 1-z)\sqrt{z(1-z)}dz^{-1/2},
$$
where we have used the property $\lambda(-\tau^{-1}) = 1-\lambda(\tau)$. The multivalued function $F(\tfrac{1}{2},\tfrac{1}{2},1; z)$ is the usual hypergeometric function solving~\eqref{gausshpg} with $a=b=\tfrac{1}{2}$ and $c=1$.
%

Let $j = xu + y\alpha u$ be an element of the jet bundle over the point $p\in \U$, written in the standard basis where $u=dz^{-1/2}$ and $\alpha = (z(1-z))^{-1} dz$.  We may extend it uniquely to a multisection $\xi = a \eta_1 + b\eta_2$ of $\can[\U]^{-1/2}$ in the kernel of $D$, where the constants $a, b$ are given by the quotients of Wronskians 
\begin{equation*}
	a = \frac{\phantom{^1_p\eta_1}j\wedge j^1_p \eta_2}{j^1_p\eta_1\wedge j^1_p\eta_2},\hspace{3em} b= \frac{j^1_p\eta_1\wedge j\phantom{^1_p\eta_2}}{j^1_p\eta_1\wedge j^1_p\eta_2}.
\end{equation*}
The value of the source map $s([j])=q$ is then given by the unique point where $a\eta_1(q) + b\eta_2(q) = 0$.  Using the fact that $\lambda(\eta_2(z)/\eta_1(z)) = z$, we obtain 
$$ 
	s([j]) = \lambda(-\tfrac{a}{b}).
$$
Computing Wronskians with $\eta_i = f_i dz^{-1/2}$ and $\delta = z(1-z)\cvf{z}$, we obtain the groupoid source map in homogeneous coordinates $[x:y]$ for $[j]$:
$$
s([x:y]) = \lambda \left( \frac{yf_2 - x \delta f_2}{yf_1 - x\delta f_1}\right). 
$$
The apparent multivaluedness of the above expression is resolved by the constraint $s(\gid(p))=p$ on the identity bisection.

%% file: stokes1-solns.tex
Let $\X$ be a curve with effective divisor $\D$, and let $(\sE,\nabla)$ be a vector bundle of rank $r$ with a meromorphic connection whose poles are bounded by $\D$.  Traditionally, a \emph{fundamental solution} to this system is the choice of a flat basis, i.e., a flat isomorphism 
$$
\psi:\sO{\X}^r\to \sE.
$$
Of course, such an isomorphism would trivialize $(\sE,\nabla)$, and so $\psi$ is typically singular along $\D$, as well as multivalued away from $\D$.  

By viewing $(\sE,\nabla)$ as a representation of $\tshf{\X}(-\D)$, we know from~\autoref{integrep} that it integrates to a unique groupoid representation, hence a holomorphic bundle isomorphism
$$
\Psi:s^*\sE\stackrel{\cong}{\longrightarrow} t^*\sE
$$
over the source-simply connected groupoid $\FG{\X,\D}$ constructed in~\autoref{sec:unif}.  

The isomorphism $\Psi(g)$, for $g\in \FG{\X,\D}$ is, generically, the parallel transport from $s(g)$ to $t(g)$, implying the following result.
\begin{proposition}
For any fundamental solution $\psi$ as above, the expression
\begin{equation}\label{deltagrpd}
	\delta \psi = t^*\psi \circ (s^*\psi)^{-1}
\end{equation}
extends holomorphically to $\FG{\X,\D}$ and coincides with the representation $\Psi$. The apparent multivaluedness of~\eqref{deltagrpd} is resolved by requiring $\delta\psi=1$ over the identity bisection.
\end{proposition}

Note that $\Psi$ is independent of the choice of fundamental solution $\psi$.  To reconstruct $\psi$ from $\Psi$, we must choose a basis $\psi(x_0)\in\sE|_{x_0}^{\oplus r}$ over a point $x_0\in\U=\X\setminus\D$, in which case $\psi$ is determined by the condition 
$$
t|_{s^{-1}(x_0)}^*\psi = \Psi (s^*\psi(x_0)),
$$
which holds along $s^{-1}(x_0)$, the universal cover of $\U$.

\begin{example}({Rank one representations on $\AF[1]$})

Fix $p\in\AF[1]$ and $k\geq 1$ in $\ZZ$, as well as the rank one representation for $\tshf{\AF[1]}(-k\cdot p)$ on $\sO{\AF[1]}$ given, for fixed $a\in\CC$, by 
\begin{equation}\label{picrkone}
\nabla = d + az^{-k}dz
\end{equation}
in a coordinate centred at $p$.  The equation $\nabla\psi = 0$ has (singular, multivalued) solutions generated by $\psi_k$, where  
\begin{equation}\label{badsol}
	\begin{aligned}
	\psi_1 &= z^{-a},\\
	\psi_k &= \exp(\tfrac{az^{-(k-1)}}{k-1}),\hspace{2ex}k>1.
	\end{aligned}
\end{equation}
The groupoid representation $\Psi_k$ of $\Sto{k}$ integrating the connection~\eqref{picrkone}, written in the coordinates from~\autoref{thm:twfunline}, is given by
\begin{equation}\label{grpdrep}\begin{aligned}
	\Psi_1 &= t^*\psi_1 (s^*\psi_1)^{-1} = e^{-a u},\\
	\Psi_k &= t^*\psi_k (s^*\psi_k)^{-1} = e^{-aS_k},
\end{aligned}\end{equation}
where $S_k,\ k>1$, is the analytic additive function $\Sto{k}\to\CC$ given by 
$$
S_k(z,u) = \frac{1-e^{-u(k-1)z^{k-1}}}{(k-1)z^{k-1}}.
$$
These representations descend to $\Pair{\AF[1],k\cdot p}$ via the homomorphism~\eqref{grpexp} precisely when they have trivial monodromy.  This situation occurs for $k=1$ when $a$ is integral, and for $k>1$ without condition on $a$.  The resulting representations $\Psi_k'$ of $\Pair{\AF[1],k\cdot p}$ are:
\begin{equation}\label{grpdrep2}
\Psi_1'= (1+u)^{-a},\hspace{2em} \Psi_k' = e^{-aS'_k},
\end{equation}
where $S'_k,\ k>1$, is given by the smooth additive function
\begin{equation}\label{essfn}
	S'_k(z,u) = \frac{1-(1+uz^{k-1})^{-(k-1)}}{(k-1)z^{k-1}}.
\end{equation}

These examples demonstrate the fact that, unlike traditional fundamental solutions~\eqref{badsol}, the representations~\eqref{grpdrep}, \eqref{grpdrep2} are smooth, single-valued, and everywhere invertible when written on the appropriate groupoid.\qed
\end{example}

\subsection{Summation of divergent series}

In attempting to solve an ordinary differential equation using a power series expansion, one encounters the phenomenon that if the origin is an irregular singular point, the resulting series will generally have zero radius of convergence.  A simple example is the inhomogeneous equation $z^2f' = f-z$, whose solution in formal power series centred at the singular point $z=0$ is
\begin{equation}\label{firstdivex}
\hat f(z)= \sum_{n=0}^\infty n!\,z^{n+1}.	
\end{equation}
Such divergent power series are called formal solutions, and indeed they are solutions of the system when restricted to the formal neighbourhood of the singular point.  
The formal method for analyzing ODEs in general, extensively developed by Sibuya~\cite{MR0096016} (see also Wasow~\cite{MR0203188} and Balser~\cite{MR1722871}), has led to an extensive library of analytic methods, including the summation procedures of Borel, \'Ecalle, and Ramis, by which the formal solution is used to obtain actual solutions, which in certain sectors around the singular point have asymptotic expansions given by the formal series.     

We now present an alternative way to understand why and how formal series may be used to construct actual solutions, making essential use of the groupoids introduced in~\autoref{sec:constr}.  
Suppose that $(\sE_1,\nabla_1)$ and $(\sE_2,\nabla_2)$ are two representations of the Lie algebroid $\sA=\tshf{\X}(-\D)$, and that we have a \emph{formal isomorphism}
\[
\xymatrix{(\hsE_1,\hnabla_1)\ar[r]^{\hphi} & (\hsE_2,\hnabla_2)},
\]
meaning an isomorphism of the restrictions of the representations to the formal neighbourhood\footnote{Also known as the \emph{formal completion} of $\D$ in $\X$.} $\hX$ of $\D$ in $\X$. Recall that ``functions'' on $\hX$ are represented by formal power series in coordinates centred at the points of $\D$, and that the restriction of sections of vector bundles to $\hX$ simply means expanding them in Taylor series.
By~\autoref{integrep}, each representation $(\sE_i,\nabla_i)$ integrates to a representation $\Psi_i$ of the groupoid $\G=\FG{\X,\D}$, which may then be restricted to the formal neighbourhood $\hG$ of the full subgroupoid $\G|_\D$ of $\G$ over $\D$ (which coincides with $\G|_{\hX}$). 
Call the resulting formal representations $\hPsi_1, \hPsi_2$.
By the same reasoning as in the proof of~\autoref{integrep},  the fact that $\hphi$ is a formal isomorphism of algebroid representations implies that we have the following commutative diagram:
\[
\xymatrix{
s^*\hsE_1 \ar[r]^{\hPsi_1} \ar[d]_{s^*\hphi} & t^*\hsE_1 \ar[d]^{t^*\hphi} \\
s^*\hsE_2 \ar[r]^{\hPsi_2} & t^*\hsE_2
}
\]
The key aspect of this diagram is that while $\hphi$ is purely formal, it intertwines formal representations $\hPsi_1, \hPsi_2$ that do extend holomorphically to $\G$.  As a result, the knowledge of $\hphi$ and $\Psi_1$ is enough to reconstruct a convergent power series for $\Psi_2$ and vice-versa.  We record this observation as follows.
\begin{theorem}\label{condediv}
Let $\hphi$ be a formal isomorphism between $\sA$-representations $(\sE_1,\nabla_1)$, $(\sE_2,\nabla_2)$ along $\hX$, so that 
\[
\hphi\circ \hnabla_1 = \hnabla_2 \circ \hphi.
\]
If $\Psi_1$ is the integration of $\nabla_1$ to a groupoid representation, then the formal isomorphism along $\hG$ given by
\begin{equation}\label{pwrtsfm}
\Sigma(\hPsi_1,\hphi) = (t^*\hphi) \cdot \hPsi_1 \cdot (s^*\hphi)^{-1}
\end{equation}
coincides with the restriction to $\hG$ of the $\G$-representation $\Psi_2$ integrating $(\sE_2,\nabla_2)$.  In particular, $\Sigma(\hPsi_1,\hphi)$, when represented in coordinates centred at a point in $\hG$, is a power series in two variables with a nontrivial domain of convergence.
\end{theorem}
\begin{remark}
From~\eqref{pwrtsfm}, it is clear that the convergent series $\Sigma(\hPsi_1,\hphi)$ is determined by $\hPsi_1$ and $\hphi$ in such a way that its terms of degree $k$ depend only on the finitely many terms of $\hPsi_1$ and $\hphi$ of degree at most $k$.   \qed
\end{remark}
\begin{example}
The inhomogeneous differential equation $z^2f'=f-z$ with formal solution~\eqref{firstdivex} may be recast as the problem of finding an isomorphism 
\begin{equation*}
 	\phi = \mat{1 & f\\ & 1}
\end{equation*} 
between the rank two connections $\nabla_1, \nabla_2$ given by 
\begin{equation}
 	\nabla_1=d + \mat{-1 & \\ & 0}z^{-2}dz,\hspace{3em}
		\nabla_2=d + \mat{-1 & z\\ & 0}z^{-2}dz.
 \end{equation} 
Indeed, the equation $\phi\circ\nabla_1 = \nabla_2\circ\phi$ holds if and only if $z^2f'=f-z$.  With the formal solution~\eqref{firstdivex} in hand, according to~\autoref{condediv} we need only solve the diagonal system $\nabla_1$ in order to obtain convergent series solutions to $\nabla_2$.  

Both $\nabla_1, \nabla_2$ are representations of $\tshf{\AF[1]}(-2\cdot 0)$, and so the corresponding groupoid representations $\Psi_1, \Psi_2$ are defined on $\Sto{2}$.  Although $\nabla_2$ has nontrivial monodromy, for convenience we will use the groupoid $\Pair{\AF[1],2\cdot p}$, over which $\Psi_2$ is multivalued.  Also, instead of the coordinates $(z,u)$ used in~\autoref{paircords}, we use $(z,\mu)$, where $\mu$ is the additive function given in~\eqref{essfn} for $k=2$.  Thus, $\mu = u(1+zu)^{-1}$, which is well-defined since $1+zu\neq 0$.  In these coordinates, we have source and target maps 
\begin{equation}
	s(z,\mu)=z,\hspace{3em} t(z,\mu) = z(1-z\mu)^{-1},
\end{equation}
and the representation $\Psi_1$ integrating the diagonal connection is given by 
\begin{equation}
	\Psi_1=\mat{e^\mu & \\ & 1}.
\end{equation}
Conjugating by the formal solution as in~\eqref{pwrtsfm}, we obtain the power series
\begin{equation}
	\hPsi_2 = \mat{1 & t^*\hf\\ & 1}\mat{e^\mu & \\ & 1}\mat{1 & -s^*\hf\\  & 1} = \mat{e^\mu & t^*\hf - e^\mu s^*\hf\\ &1}.
\end{equation}
Using the divergent series $\hf = \sum_{n=0}^\infty n!\,z^{n+1}$ and computing in power series about $(z,\mu)=(0,0)$, we obtain
\begin{equation}\label{convexpand}\begin{aligned}
	t^*\hf-e^\mu s^*\hf &= -\sum_{i=0}^\infty\sum_{j=0}^\infty \frac{z^{i+1}\mu^{i+j+1}}{(i+1)(i+2)\cdots(i+j+1)},
\end{aligned}\end{equation}
which is a convergent power series in two variables for the representation $\Psi_2$.  

In this example, we may solve the system using special functions; in this way we obtain the fact that~\eqref{convexpand} is a Taylor expansion of
\begin{equation}
 	\rho(z,u)=e^\frac{z\mu-1}{z}\left(\mathrm{Ei}(\tfrac{1-z\mu}{z})-\mathrm{Ei}(\tfrac{1}{z})\right)
 \end{equation}
 about the removable singularity at $(0,0)$.  The multivaluedness of $\rho$ is resolved near $(0,0)$ by the requirement $\rho(z,0)=0$.  \qed
\end{example}

\begin{example}[\airylabel]\label{airylast}
Recall from~\autoref{ex:airy} that the Airy system on $\PP[1]$ has a single pole of degree three at infinity. For clarity of exposition, we perform a modification by $h=\mathrm{diag}(1, -z^{-1})$, so that near the pole we have the system
$$
h\nabla^\Airy h^{-1} =  d + \begin{pmatrix}
0 & z \\
1 & 0
\end{pmatrix}z^{-3}dz,
$$
which we continue to call the Airy system, and which has fundamental solution
$$
\begin{pmatrix}
\Ai\phantom{'}(z^{-1}) & \Bi\phantom{'}(z^{-1}) \\
\Ai'(z^{-1}) & \Bi'(z^{-1})
\end{pmatrix}.
$$
As observed by Stokes in his 1847 paper~\cite[Eq.~14]{Stokes1847} (see also Olver~\cite[Appendix]{MR1429619}), the modification $H=\mathrm{diag}(z^{1/4},z^{-1/4})$,
composed with the formal (divergent) isomorphism 
$$ 
\hphi = \mat{l(-\zeta)&\phantom{-}l(\zeta)\\-m(-\zeta)&m(\zeta)},\hspace{2em} \zeta = \tfrac{3}{2} z^\frac{3}{2},
$$
where $l(\zeta)=\sum_{n=0}^\infty l_n\zeta^n$ and $m(\zeta)=\sum_{n=0}^\infty m_n\zeta^n$ are defined by
$$
l_n = \frac{2^n\Gamma(3n+\frac{1}{2})}{3^{3n}(2n)!\Gamma(\frac{1}{2})},\hspace{3em} m_n = -\frac{6n+1}{6n-1}l_n,
$$
takes the Airy system to a diagonal system with fundamental solution
$$
\psi_0 = 
\begin{pmatrix}e^{ -\frac{2}{3} z^{-\frac{3}{2}}   } & \\ & e^{ \frac{2}{3} z^{-\frac{3}{2}}   }
\end{pmatrix}.
$$
Consequently, the Airy representation on the groupoid $\Pair{\AF[1],3\cdot 0}$ centred at the pole is formally given by
$$
\hPsi_{\Airy}=t^*(H\hphi) \cdot (t^*\psi_0 s^*\psi_0^{-1})\cdot  s^*(H\hphi)^{-1}.
$$
In the coordinates $(z,u)$ from~\autoref{paircords} in which $s=z$ and $t=z(1 + z^2u)$ are the groupoid source and target maps, the Airy representation~(cf.~\eqref{airyunivsol})
\begin{equation}\label{airyatinfty}
\Psi_{\Airy} = 
\begingroup
\renewcommand*{\arraystretch}{1.2}
\pi\begin{pmatrix}
\Ai\phantom{'}(\frac{1}{t})\Bi'(\frac{1}{s}) - \Bi\phantom{'}(\frac{1}{t})\Ai'(\frac{1}{s}) & \
  - \Ai\phantom{'}(\frac{1}{t})\Bi(\frac{1}{s}) + \Bi\phantom{'}(\frac{1}{t}) \Ai(\frac{1}{s}) \\
 \Ai'(\frac{1}{t})\Bi'(\frac{1}{s}) - \Bi'(\frac{1}{t})\Ai'(\frac{1}{s})& \ 
 -\Ai'(\frac{1}{t})\Bi(\frac{1}{s})+\Bi'(\frac{1}{t})\Ai(\frac{1}{s})
\end{pmatrix}
\endgroup
\end{equation}
can therefore be expressed by the following power series in $z$ and $u$:
\begin{equation*}\label{pwrseriesmess}
\tfrac{1}{2}
\begin{psmallmatrix}
	t^\frac{1}{4} & \\ & t^{-\frac{1}{4}}
\end{psmallmatrix}
\begin{psmallmatrix}
	\phantom{-}l(-\frac{3}{2}t^{\frac{3}{2}})&l(\frac{3}{2}t^{\frac{3}{2}})\\-m(-\frac{3}{2}t^{\frac{3}{2}})&m(\frac{3}{2}t^{\frac{3}{2}})
\end{psmallmatrix}
\begin{psmallmatrix}
	e^{\frac{2}{3}(s^{-\frac{3}{2}}-t^{-\frac{3}{2}})} & \\ & 
	e^{\frac{2}{3}(t^{-\frac{3}{2}}-s^{-\frac{3}{2}})}
\end{psmallmatrix}
\begin{psmallmatrix}
	m(\frac{3}{2}s^{\frac{3}{2}})&-l(\frac{3}{2}s^{\frac{3}{2}})\\m(-\frac{3}{2}s^{\frac{3}{2}})&\phantom{-}l(-\frac{3}{2}s^{\frac{3}{2}})
\end{psmallmatrix}
\begin{psmallmatrix}
	s^{-\frac{1}{4}} & \\ & s^{\frac{1}{4}}
\end{psmallmatrix}.
\end{equation*}
Expanding the product above and exhibiting terms of total degree six or less, we obtain
\begin{align}
\begingroup
\renewcommand*{\arraystretch}{1.2}
\begin{pmatrix}
1 + \tfrac{1}{2}u^2z - \tfrac{7}{6}u^3z^3+\tfrac{1}{24}u^4z^2 + \cdots & -uz + u^2z^3-\tfrac{1}{6}u^3z^2 +\cdots\\
-u + \tfrac{3}{2}u^2z^2-\tfrac{1}{6}u^3z +\cdots & 1 + \tfrac{1}{2}u^2z-\tfrac{4}{3}u^3z^3+\tfrac{1}{24}u^4z^2 +\cdots
\end{pmatrix},
\endgroup
\end{align}
verifying the fact that~\eqref{airyatinfty} has a removable singularity at $(z,u)=(0,0)$.
\qed
\end{example}